\documentclass[a4paper,10pt]{amsart}
	\usepackage[english]{babel}
	\usepackage{amsmath,amssymb,amsthm}
	\usepackage[pdftex]{color}
	\usepackage[bookmarks=true,hyperindex,pdftex,colorlinks, citecolor=MidnightBlue,linkcolor=MidnightBlue, urlcolor=MidnightBlue]{hyperref}
	\usepackage[dvipsnames]{xcolor}
	\usepackage{enumerate}
	\usepackage{lineno}
	\usepackage[edges]{forest}
	\usepackage{mathtools}
	\usepackage{graphicx}
	\usepackage{subfig}
	\usepackage{url}
	\usepackage{hyperref}
	
	\usetikzlibrary{decorations.markings}
	\tikzset{neg/.style={
			decoration={markings,
				mark= at position 0.5 with {
					\node[transform shape] (tempnode) {$\setminus$};
				}
			},
			postaction={decorate}
	}}
	
	\newtheorem{theorem}{Theorem}[section]
	\newtheorem{proposition}[theorem]{Proposition}
	\newtheorem{corollary}[theorem]{Corollary}
	\newtheorem{lemma}[theorem]{Lemma}
	\newtheorem{example}[theorem]{Example}
	
	\theoremstyle{definition}
	
	\newtheorem*{definition*}{Definition}
	\newtheorem*{proposition*}{Proposition}
	\newtheorem*{theorem*}{Theorem}
	\newtheorem*{corollary*}{Corollary}
	\newtheorem*{example*}{Example}
	\newtheorem*{problem*}{Problem}

	\theoremstyle{remark}
	\newtheorem{remark}[theorem]{Remark}

	\newcommand*{\prob}[1]{\mathbb{P}\left\{ #1 \right\}}

	\newcommand*{\Var}{\operatorname{Var}}
	\newcommand*{\Cov}{\operatorname{Cov}}
	\newcommand*{\Val}{\operatorname{val}}
	\newcommand*{\stcomp}[1]{{#1}^{\mathsf{c}}}

\begin{document}
		
		\title{A local central limit theorem for random walks on expander graphs}
		
		\author[Chiclana]{Rafael Chiclana}
		
		\author[Peres]{Yuval Peres}
		
		\address{Kent State University, Kent, Ohio, United States}
		\address{Beijing Institute of Mathematical Sciences and Applications, Beijing, China}
		\email{rchiclan@kent.edu, yperes@gmail.com}
		

		\begin{abstract}
			There is a long history of establishing central limit theorems for Markov chains. Quantitative bounds for chains with a spectral gap were proved by Mann and refined later. Recently, rates of convergence for the total variation distance were obtained for random walks on expander graphs, which are often used to generate sequences satisfying desirable pseudorandom properties. We prove a local central limit theorem with an explicit rate of convergence for random walks on expander graphs, and derive an improved bound for the total variation distance.
		\end{abstract}

		\maketitle
		
		\thispagestyle{plain}
		\section{Introduction}\label{sec:introduction}
		Given $\lambda<1$, a graph is considered to be a \textit{$\lambda$-expander} when the absolute value of all the eigenvalues of its transition matrix except $1$ are bounded above by $\lambda$. Expander graphs have a wide range of applications in areas such as derandomization, complexity theory, and coding theory (see \cite{hoory2006expander}). In particular, random walks on expander graphs are typically used to generate sequences satisfying desirable pseudorandom properties. They serve as an efficient replacement of $t$ independent sample vertices chosen uniformly at random. It is natural to study then how good of a replacement these sequences are, or equivalently, to measure the randomness of random walks on expander graphs. More precisely, consider a \textit{balanced labelling} on a regular graph $G=(V,E)$, that is, a map $\Val\colon V \longrightarrow \{0,1\}$ with $\sum_{v \in V} \Val(v)=|V|/2$. Given a test function $f\colon\{0,1\}^t \longrightarrow \mathbb{R}$, we compare $f(\Val(v_0),\ldots,\Val(v_{t-1}))$ when the vertices $v_0,\ldots,v_{t-1}$ are sampled either from a random walk, or independently and uniformly at random. This problem was studied by Guruswami and Kumar in \cite{guruswami2021pseudobinomiality} for sticky random walks, and later on, by Cohen, Peri, and Ta-Shma in \cite{cohen2021expander} for general expander graphs. Significant progress can be found also in Cohen et al. \cite{cohen2022expander}, Golowich-Vadhan \cite{golowich2022pseudorandomness}, and Golowich \cite{golowhich2022a}. In this paper, we focus on the asymptotic behavior of $f(\Val(v_0),\ldots,\Val(v_{t-1}))$ as the size of the sample $t$ grows. Our results answer Question 3 in \cite{cohen2021expander} and Questions 2 and 3 in \cite{cohen2022expander}. Moreover, we improve the bound of the main result of \cite{golowhich2022a} for bounded degree graphs.
		
		Throughout the paper, we assume that all graphs are finite and connected. We write $\mathcal{N}(\mu,\sigma^2)$ for a normal distribution with mean $\mu$ and variance $\sigma^2$, and $\phi$ for the density function of $\mathcal{N}(0,1)$. Local central limit theorems for Markov chains are known as early as the work of Kolmogorov \cite{kolmogorov1949a}, and the contributions due to Nagaev \cite{nagaev1957some} and \cite{nagaev1962more}, who initiated the study of Markov chains by spectral methods. This topic has been widely studied over the last years. Our main result gives a local central limit theorem for the random walk on expander graphs with a uniform rate of convergence.
		
		\begin{theorem}\label{theo:main}
			Let $G=(V,E)$ be a $d$-regular $\lambda$-expander graph with $\lambda<1$. Fix a balanced labelling $\operatorname{val}\colon V \longrightarrow \{0,1\}$, let $(X_i)$ be the simple random walk on $G$ with uniform initial distribution, and let $Z_t= \sum_{i=0}^{t-1} \Val(X_i)$ and $\sigma^2=\lim_{t\to \infty} \Var(Z_t)/t$. There is a constant $C_1(\lambda,d)$ depending only on $\lambda$ and $d$ such that
			\[ \left |\prob{Z_t =k} - t^{-1/2}\sigma^{-1}\phi\left (\frac{k-t/2}{t^{1/2}\sigma}\right )\right | \leq C_1(\lambda,d) \frac{1}{t} \quad \forall \, k \in \mathbb{Z} \quad \forall\, t \in \mathbb{N}. \]
		\end{theorem}
		We obtain this from a general local central limit theorem for Markov chains, given in Section \ref{sec: general lclt}.
	
		\begin{remark}\label{remark uniform}
			Every $d$-regular connected graph that is not bipartite is a $\lambda$-expander for some $\lambda<1$. It follows from Theorem \ref{theo:main} that the local central limit theorem holds for any of such graphs. However, the time $t$ until $Z_t$ and $\mathcal{N}(t/2,t\sigma^2)$ are close might depend on the number of vertices of $G$. The fundamental observation in Theorem \ref{theo:main} is that the dependence on the size of the graph disappears when $G$ has an absolute spectral gap and bounded degree.
		\end{remark}
		
		Write $(U_i)$ for a sequence of independent vertices of $G$ chosen uniformly at random. We are particularly interested in the total variation distance between the Hamming weights $Z_t=\sum_{i=0}^{t-1} \Val(X_i)$ and $B_t=\sum_{i=0}^{t-1} \Val(U_i)$, denoted\footnotemark{} $\|Z_t-B_t\|_{TV}$. \footnotetext{The total variation distance is defined between measures. We abuse notation and identify $\|Z_t-B_t\|_{TV}$ with $\|\mathcal{L}(Z_t) - \mathcal{L}(B_t)\|_{TV}$, where $\mathcal{L}(\cdot)$ stands for the law of a random variable.} This distance measures the best distinguishing probability a symmetric function can achieve on $(\Val(X_i))_{i=0}^{t-1}$ and $(\Val(U_i))_{i=0}^{t-1}$ (see Proposition 4.5 in \cite{levin2017markov}).

		The asymptotic behavior of $Z_t$ is determined by Theorem \ref{theo:main}.  In fact, most of the mass of $Z_t$ is concentrated in an interval of length $\sqrt{t\log t}$ around its mean (see Theorem 2.1 in \cite{gillman1998a}). Therefore, a local central limit theorem for $Z_t$ implies convergence in total variation distance to a discretized normal distribution. Indeed, since Theorem \ref{theo:main} gives a convergence rate of $O(1/t)$, this simple argument gives a rate of convergence of $O(\sqrt{\log t}/\sqrt{t})$ for the total variation distance. A sharper analysis allows to improve this bound to $O(\log(t)^{1/4}/\sqrt{t})$. In an earlier version of this work \cite{chiclana2022av1}, we demonstrated convergence in total variation distance without bounding the rate of convergence. Subsequently, Golowich presented in \cite{golowhich2022a} a bound similar to (\ref{eq: cor 1.3}), where $\log(t)^{1/4}$ is replaced by $\log(t)^{\eta_1 \log \log t + \eta_2}$ for some constants $\eta_1$, $\eta_2$ (see (\ref{eq: golowich})).
		
		There are different natural ways of discretizing the normal distribution. For convenience, we will consider $N_d(\mu,\sigma^2)$ with probability distribution $f_{N_d(\mu,\sigma^2)}$ given by 
		\begin{equation}\label{eq:disc normal}
			f_{N_d(\mu,\sigma^2)}(k)=\frac{1}{D(\mu,\sigma^2)} \sigma^{-1}\phi\left (\frac{k-\mu}{\sigma}\right ) \quad \forall \, k \in \mathbb{Z},
		\end{equation}
		where $D(\mu,\sigma^2)=\sum_{k \in \mathbb{Z}} \sigma^{-1}\phi\left (\frac{k-\mu}{\sigma}\right )$ is a normalizing constant.
		
		\begin{corollary}\label{cor:main}
			Let $G=(V,E)$ be a $d$-regular $\lambda$-expander graph with $\lambda<1$. Fix a balanced labelling $\operatorname{val}\colon V \longrightarrow \{0,1\}$, let $(X_i)$ be the simple random walk on $G$ with uniform initial distribution, and let $Z_t= \sum_{i=0}^{t-1} \Val(X_i)$ and $\sigma^2=\lim_{t\to \infty} \Var(Z_t)/t$. There is a constant $C_2(\lambda,d)$ depending only on $\lambda$ and $d$ such that
			\begin{equation}\label{eq: cor 1.3} \left \|Z_t - N_d(t/2,t\sigma^2) \right \|_{TV}\leq C_2(\lambda,d) \frac{\log (t)^{1/4}}{\sqrt{t}} \quad \forall \, t \geq 2.
			\end{equation}
		\end{corollary}
	
		On the other hand, the asymptotic behavior of $(B_t)$ is well known. Indeed, $\Val(U_i)$ is a Bernoulli random variable with parameter $p=\frac{1}{2}$. Hence, $B_t$ follows a binomial $\operatorname{Bin}(t,\frac{1}{2})$. Since binomial distributions are concentrated around their means (see Lemma 8.1 in \cite{penrose2011local}), the classic local central limit theorem (see Lemma 5 in \cite{zolotukhin2018on}) implies that $B_t$ converges in total variation distance to a discretized normal distribution. More precisely,
		\[ \lim_{t \to \infty} \|B_t- N_d(t/2,t/4)\|_{TV}=0.\]
		In view of Corollary \ref{cor:main}, we deduce that
		\[
			\lim_{t \to \infty} \|B_t - Z_t\|_{TV} = \|N_d(t/2,t/4) - N_d(t/2,t\sigma^2)\|_{TV}.
		\]
		Even though both $Z_t$ and $B_t$ converge to discretized normal distributions with mean $\frac{t}{2}$, their variances may not be the same. Therefore, the ability of a random walk $(X_i)$ on an expander graph to fool all symmetric functions as $t$ grows is measured by the difference between the variances of $Z_t$ and $B_t$. This difference can be bounded using the following formula for the variance of $Z_t$, which is written in terms of the eigenvalues $\lambda_j$, their normalized eigenvectors $f_j$ (see discussion previous to (\ref{eq:spectral form Pf})), and the labelling $\Val$. We give it also for unbalanced labellings since it will be useful when extending our main results. Balanced labellings correspond to taking $\alpha=1/2$ below.
		
		\begin{proposition}\label{prop: variance}
			Let $G$ be a $d$-regular graph with $n$ vertices. Let $(X_i)$ be the simple random walk on $G$ with uniform initial distribution $\pi$, and consider $Z_t= \sum_{i=0}^{t-1} \Val(X_i)$. For a labelling $\operatorname{val}\colon V \longrightarrow \{0,1\}$ with $\mathbb{E}_\pi(\Val)=\alpha \in (0,1)$, we have
			\begin{equation}\label{eq:var formula} 
				\Var(Z_t)= \alpha(1-\alpha)t +  2\alpha^2 \sum_{k=1}^{t-1}(t-k)\sum_{j=2}^n\langle \pi_B,f_j\rangle^2\lambda^{k}_j,
			\end{equation}
			where $B=\{x \in V \colon \Val(x)=1\}$ and $\pi_B$ is the uniform distribution on $B$. In particular, if $G$ is a $\lambda$-expander then we have
			\begin{equation}\label{eq:var bound} 
				\left |\Var(Z_t)-\alpha(1-\alpha)t\right |\leq 2\alpha(1-\alpha)t\frac{\lambda}{1-\lambda}.
			\end{equation}
		\end{proposition}
		
		The \textit{sticky random walk} with parameter $p \in (-1,1)$ is a Markov chain $(Q_i)$ on $\{0,1\}$ defined as follows. The initial state is chosen uniformly at random. At each step, the chain stays at the same state with probability $\frac{1+p}{2}$, and switches states with probability $\frac{1-p}{2}$. This simple chain can be seen as a simplified version of general random walks on expander graphs. Although a sequence of bits generated using the random walk may not fool all symmetric functions as $t$ grows, it serves as a replacement for a sample of bits obtained from the sticky random walk. The following result is an immediate consequence of Corollary \ref{cor:main} and a similar result for the sticky random walk (see Lemma \ref{lem: sticky tvd}).
		
		\begin{theorem}\label{theo: sticky}
			Let $G=(V,E)$ be a $d$-regular $\lambda$-expander graph with $\lambda<1$. Fix a balanced labelling $\Val \colon V \longrightarrow \{0,1\}$, let $(X_i)$ be the simple random walk on $G$ with uniform initial distribution, $Z_t=\sum_{i=0}^{t-1} \Val(X_i)$, and $\sigma^2=\lim_{t\to \infty} \Var(Z_t)/t$. Let $(Q_i)$ be the sticky random walk on $\{0,1\}$ with parameter $p=\frac{4\sigma^2-1}{4\sigma^2+1}$ and $R_t=\sum_{i=0}^{t-1} Q_i$. There is a constant $C_3(\lambda,d)$ depending only on $\lambda$ and $d$ such that
			\begin{equation}\label{eq: sticky} \|Z_t-R_t\|_{TV} \leq C_3(\lambda,d) \frac{\sqrt{\log t}}{\sqrt{t}} \quad \forall \, t\geq 2.
			\end{equation}
		\end{theorem}
		We believe that power of the logarithm in (\ref{eq: sticky}) can be improved with additional work.

		The rest of the paper is organized as follows. In Section \ref{sec:preliminaries}, we introduce notation and definitions that will be used throughout the paper. In Section \ref{sec:previous work}, we discuss significant previous work of several authors on this topic. In section \ref{sec: general lclt}, we present a general local central limit theorem for Markov chains. Our main result follows as a particular application of it. Section \ref{sec:main result} is dedicated to prove Theorem \ref{theo:main}. In Section \ref{sec:variance}, we prove Proposition \ref{prop: variance} and Corollary \ref{cor:main}. Section \ref{sec:sticky} is devoted to prove Theorem \ref{theo: sticky}. In Section \ref{sec:extension}, we extend our main results to unbalanced labellings. Finally, we dedicate Section \ref{sec:examples} to prove Example \ref{ex: K4} and Example \ref{ex: uniform}.

		\section{Previous work}\label{sec:previous work}
		In the recent paper \cite{guruswami2021pseudobinomiality}, among other results, Guruswami and Kumar showed that the total variation distance between the Hamming weight of the sticky random walk with parameter $p$ and the binomial distribution is $\Theta(p)$. As \cite{cohen2021expander} states, ``a major open problem they raise is whether the same is true for random walks on expander graphs''. This problem has been studied very recently by several authors. We discuss here the most significant advances on the matter.
		
		Cohen, Peri, and Ta-Shma present in \cite{cohen2021expander} a Fourier-Analytic approach to study random walks on expander graphs. Their main result states that the SRW on $\lambda$-expander graphs ``fools" symmetric functions for small values of $\lambda$. To be more precise, let $G$ be a $d$-regular $\lambda$-expander graph, $(X_i)$ the SRW on $G$ with uniform initial distribution $\pi$, and $(U_i)$ a sequence of independent vertices of $G$ chosen uniformly at random.  Theorem 1.1 in \cite{cohen2021expander} states that for any $t\in \mathbb{N}$, any symmetric function $f \colon\{ 0,1\}^t \longrightarrow \{0,1\}$, and any balanced labelling $\Val \colon V \longrightarrow\{0,1\}$, we have
		\[ |\mathbb{E}_\pi(f(\Val(X_0),\ldots,\Val(X_{t-1}))) - \mathbb{E}_\pi(f(\Val(U_0),\ldots,\Val(U_{t-1})))| = O(\lambda \cdot \log^{3/2}(1/\lambda)).
		\]
		This result can be rewritten in terms of the total variation distance between the Hamming weights $Z_t=\sum_{k=0}^{t-1}\Val(X_k)$ and $B_t=\sum_{k=0}^{t-1} \Val(U_k)$ as follows.
		\begin{equation}\label{eq:cohen main} 
			\| Z_t - B_t\|_{TV}= O(\lambda \cdot \log^{3/2}(1/\lambda)).
		\end{equation}
		The authors then propose several open questions. First, they ask if (\ref{eq:cohen main}) holds for unbalanced labellings. They also ask whether the above bound is sharp. These questions are addressed by Theorem 3 in \cite{cohen2022expander}. It states that for any labelling $\Val\colon V \longrightarrow\{1,-1\}$ with $\mathbb{E}_\pi(\Val)=\alpha \in (-1,1)$ and $0<\lambda<\frac{1-|\alpha|}{128e}$ we have
		\[
			\|Z_t-B_t\|_{TV}\leq \frac{124}{\sqrt{1-|\alpha|}}\lambda.
		\]
		An equivalent bound is achieved by Corollary 2 in \cite{golowich2022pseudorandomness}, which also provides interesting bounds for the tails of the distributions. Moreover, Corollary 4 in \cite{golowich2022pseudorandomness} extends these bounds from binary to arbitrary labellings. Finally, the authors show that the dependence on $\lambda$ in the above results is sharp up to a constant (see Theorem 5 in \cite{golowich2022pseudorandomness}).
		
		Second, while (\ref{eq:cohen main}) shows that the total variation distance between $Z_t$ and $B_t$ vanishes with $\lambda$, it leaves open the possibility that a better convergence exists, namely, for some fixed $\lambda$ the total variation distance goes to $0$ as the size of the sample $t$ grows. This is the case for some well-known symmetric test functions, such as $\operatorname{AND}$, $\operatorname{OR}$, and $\operatorname{PARITY}$, where the error decreases exponentially with $t$ (see \cite{cohen2021expander} for details), and $\operatorname{MAJ}$, where the error goes down polynomially with $t$ (see Theorem 4.6 in \cite{cohen2021expander}). This question is addressed by Theorem 1 in \cite{cohen2022expander}, which shows that for every $\lambda$ there is a $\lambda$-expander graph and a balanced labelling $\Val \colon V \longrightarrow \{1,-1\}$ such that
		\[ \|Z_t - B_t\|_{TV}=\Theta(\lambda) \quad \forall \, t \in \mathbb{N}.
		\]
		However, this estimate is obtained using Cayley graphs over Abelian groups, which cannot provide constant degree expanders. An open question that the authors in \cite{cohen2022expander} propose is whether a similar bound holds for constant degree graphs. They also ask about the existence of a family of expander graphs that fools all symmetric functions with error going down to zero as the length of the walk $t$ grows, independently of the chosen labelling. Corollary \ref{cor:main} answers both of these questions. It shows that the total variation distance between $Z_t$ and $B_t$ converges to zero as $t$ grows if, and only if, the limits of the variances of $t^{-1/2}Z_t$ and $t^{-1/2}B_t$ are the same. The following examples, that we justify in Section \ref{sec:examples}, show that in general this is not the case. Recall that $\Var(B_t)=t/4$.
		
		\begin{example}\label{ex: K4}
			Let $K_4$ be the complete graph with 4 nodes, let $\Val$ be any balanced labelling on $K_4$, let $(X_i)$ be the simple random walk on $K_4$ with uniform initial distribution, and consider $Z_t=\sum_{i=0}^{t-1} \Val(X_i)$. Then $K_4$ is a $3$-regular $\frac{1}{3}$-expander graph and
			\[ \Var(Z_t)= \frac{1}{8}t + O(1).\]
		\end{example}
	
		\begin{example}\label{ex: uniform}
			Let $G$ be a $d$-regular $\lambda$-expander graph with $n$ vertices and let $(X_i)$ be the simple random walk on $G$. There is a balanced labelling $\Val$ on $G$ for which $Z_t=\sum_{i=0}^{t-1} \Val(X_i)$ satisfies
			\[ \Var(Z_t) \geq \frac{t}{4} + \frac{1}{2}\left (\frac{1}{d} - \frac{3}{n-1}\right )t + O(1).\]
		\end{example}
	
		In the very recent paper \cite{golowhich2022a}, Golowich presents the first result with an explicit rate for the convergence in total variation distance of $Z_t$. For $\lambda\leq 1/100$, his main result gives universal constants $\eta_1$, $\eta_2$ such that
		\begin{equation}\label{eq: golowich} \|Z_t - \mathcal{N}^t_{\sigma^2}\|_{TV} \leq \frac{\lambda}{\sqrt{t}} (1+\log t)^{ \eta_1 \log \log t +\eta_2} \quad \forall \, t \in \mathbb{N},
		\end{equation}
		where $\mathcal{N}^t_{\sigma^2}$ is a discretized normal distribution satisfying a number of axioms (see Definition 11 in \cite{golowhich2022a}). It is asked in \cite{golowhich2022a} whether the factor $(1+\log t)^{ \eta_1 \log \log t +\eta_2}$ can be removed from the above bound. Although our approach does not allow us to completely get rid of it, Corollary \ref{cor:main} provides a better convergence with respect to $t$ for bounded degree graphs. On the other hand, (\ref{eq: golowich}) gives a better convergence with respect to $\lambda$ than Corollary \ref{cor:main}.
		
		\section{Preliminaries}\label{sec:preliminaries}
		Given two probability measures $\mu$, $\nu$ on a finite or countable set $V$, their \textit{total variation distance} is
		\[\|\mu-\nu\|_{TV} = \frac{1}{2} \sum_{v \in V} |\mu(v)-\nu(v)|.\]
		
		Let $G=(V,E)$ be a graph with $n$ vertices, where $V$ is the set of vertices and $E$ the set of edges. We say that $G$ is $d$\textit{-regular} if every vertex $v \in V$ has degree $d$. Let $(X_i)$ be the \textit{simple random walk} (SRW for short) on $G$ started at a vertex chosen uniformly at random from $V$, i.e., at every step the chain goes to an adjacent vertex chosen uniformly at random. The \textit{transition matrix} of $(X_i)$ is denoted by $P$, and its \textit{stationary distribution} by $\pi$. Notice that $\pi$ is the uniform distribution on $V$ since $G$ is regular. Denote $\langle\cdot,\cdot \rangle$ the usual inner product on $\mathbb{R}^V$ given by $\langle f,g \rangle= \sum_{x \in V} f(x)g(x)$. We will also consider the inner product $\langle \cdot, \cdot\rangle_\pi$ defined by 
		$$\langle f,g \rangle_\pi = \sum_{x \in V} f(x)g(x)\pi(x) \quad \forall \, f,g\colon V \longrightarrow \mathbb{R},$$
		which induces a norm $\|\cdot\|_{2,\pi}$.
		It is well known that $P$ is a self-adjoint stochastic matrix with respect to the inner product $\langle \cdot, \cdot\rangle_\pi$, thus it has real eigenvalues $1=\lambda_1>\lambda_2\geq\ldots\geq \lambda_n\geq -1$. Write $\lambda^*=\max\{|\lambda_j|\colon j\geq 2\}$. We say that $G$ is a $\lambda$\textit{-expander} graph if $\lambda^* \leq \lambda$. The \textit{absolute spectral gap} of the chain is $1-\lambda^*$.
		The spectral theorem applied to $P$ gives an orthonormal basis of eigenvectors $(f_j)_{j=1}^n$ with respect to $\langle \cdot, \cdot\rangle_\pi$ corresponding to the eigenvalues $(\lambda_j)_{j=1}^n$. As a consequence, for any $f \colon V \longrightarrow \mathbb{R}$ we have
		\begin{equation}\label{eq:spectral form Pf} 
			P^t f (x)= \sum_{j=1}^n \langle f,f_j \rangle_\pi f_j(x)\lambda_j^t \quad \forall \, x \in V \quad \forall \, t \in \mathbb{N}.
		\end{equation}
		We refer to Lemma 12.2 in \cite{levin2017markov} for a more detailed explanation. Let $W$ be a discrete random variable on $\mathbb{Z}$. The \textit{characteristic function} (ch.f. for short) of $W$ is $\varphi_W \colon [-\pi,\pi] \longrightarrow \mathbb{C}$ given by $\varphi_W(\theta)= \mathbb{E}(e^{i\theta W})$. If $N$ is a continuous random variable on $\mathbb{R}$, then its ch.f. is $\varphi_N \colon \mathbb{R} \longrightarrow \mathbb{\mathbb{C}}$ given by $\varphi_N(\theta)=\mathbb{E}(e^{i\theta N})$. 
		
		\section{A general local central limit theorem for Markov chains}\label{sec: general lclt}
		
		In this section, we prove a general central limit theorem for Markov chains with a spectral gap. Our main result Theorem \ref{theo:main} is obtained as an application of it.
		
		Penrose and Peres proved in \cite{penrose2011local} the following useful principle: Let $(Z_t)$ be a sequence of random variables that can be decomposed (with high probability) as $Z_t=S_t+Y_t$, where
		\begin{itemize}
			\item $S_t$ and $Y_t$ are independent;
			\item $S_t$ is a sum of independent identically distributed random variables;
			\item $(Y_t)$ satisfies the central limit theorem.
		\end{itemize}
		Then $(Z_t)$ satisfies the local central limit theorem (with unspecified rate of convergence). See Theorem 2.1 in \cite{penrose2011local} for the precise statement. 
		
		The next result is of the same nature as this principle, but it provides an explicit rate of convergence. We need some preliminary notation. Let $(X_i)$ be an irreducible and aperiodic Markov chain on a finite set $V$. Write $(\lambda_i)_{i=1}^{|V|}$ for its eigenvalues, where $\lambda_1=1$. We do not assume reversibility in the next result, so $\lambda_i$ might be a complex number. We still write $\lambda^*=\max\{|\lambda_i| \colon i \geq 2\}$ and define the absolute spectral gap of the chain as $1-\lambda^*$. Let $W$ be a random variable taking integer values and pick $\eta>0$ and $\theta_0 \in (0,\pi)$. We say that $W$ is an \textit{$\eta$-nonlattice for $\theta_0$} if $|\mathbb{E}(e^{i\theta W})|\leq 1-\eta$ for any $\theta \in \mathbb{R}$ satisfying $\theta_0\leq |\theta|\leq \pi$. Clearly, $|\mathbb{E}(e^{i0W})|=1$. If $|\mathbb{E}(e^{i\theta W})|=1$ for some $\theta\neq 0$, then the distribution of $e^{i \theta W}$ must be concentrated at some point $e^{i \theta b}$, and $\prob{ W \in b+(2\pi/ \theta)\mathbb{Z}}=1$. The parameter $\eta$ quantifies how far $W$ is from behaving like this when $|\theta|$ is bigger than $\theta_0$. Given two probability measures $\pi$, $\mu_0$ on a finite or countable set $V$, let $\frac{\pi - \mu_0}{\pi}$ be the vector with entries $\frac{\pi - \mu_0}{\pi}(x) = \frac{\pi(x)-\mu_0(x)}{\pi(x)}$ for any $x \in V$.
		
		\begin{theorem}\label{theo: local CLT Markov}
			Let $(X_i)$ be an irreducible and aperiodic Markov chain on $V$ with initial distribution $\mu_0$, stationary distribution $\pi$, and absolute spectral gap $1-\lambda>0$. Given $f \colon V \longrightarrow \mathbb{Z}$, write $Z_t = \sum_{i=0}^{t-1} f(X_i)$ and $\sigma^2=\lim_{t \to \infty} \Var(Z_t)/t$. Let $S_t$ and $Y_t$ be independent random variables so that $\mathbb{E}_{\mu_0}(|Z_t -S_t-Y_t|) \leq \frac{M}{t}$ for some $M>0$, and $S_t=\sum_{j=1}^{b_t} V_j$, where $b_t \in \mathbb{N}$ and $V_j$ are independent $\eta$-nonlattice random variables for $\theta_0>0$ satisfying
			\[
				\theta_0 \leq \frac{(1-\lambda)^2\sigma^2}{2708\|f\|_{\infty}^3}.
			\]
			Then there is a constant $C_4$ so that for any $k \in \mathbb{Z}$ and $t \in \mathbb{N}$,
			\[ \left  |\prob{Z_t=k} - \frac{1}{\sigma\sqrt{t}}\phi\left (\frac{k-t\mathbb{E}_\pi(f)}{\sigma\sqrt{t}}\right )\right | \leq\left (\pi M +\frac{1}{\theta_0 \sigma^2} + \frac{C_4\|f\|_{\infty}^3}{\sigma^4(1-\lambda)^2}\left (1+\left \|\frac{\pi-\mu}{\pi}\right \|_{2,\pi}\right ) \right ) \frac{1}{t} + \frac{1}{e\eta} \frac{1}{b_t}.
			\]
		\end{theorem}
	
		Although the proof does not optimize the constant, it shows that we can take $C_4=6983$. We prove Theorem \ref{theo: local CLT Markov} first in the special case where $Z_t=S_t+Y_t$, where one can take $M=0$, and then in full generality.
		
		\begin{lemma}\label{lem: Z_n=S_n+Y_n}
			Theorem \ref{theo: local CLT Markov} holds in the case where $Z_t=S_t+Y_t$.
		\end{lemma} 
		
		\begin{proof}
			First, assume that $\mathbb{E}_\pi(f)=0$. Fix $t \in \mathbb{N}$ and let $\varphi_{t\sigma^2}$ be the ch.f. of a normal $\mathcal{N}(0,t\sigma^2)$, that is, $\varphi_{t\sigma^2}(\theta ) = e^{-t\sigma^2\theta^2/2}$. Also, let $\varphi_Z$, $\varphi_S$, $\varphi_Y$, and $\varphi_{V_j}$ be the ch.f.'s of $Z_t$, $S_t$, $Y_t$, and $V_j$, respectively. The inversion formula (Theorem 3.3.14 in \cite{durret2019probability}) gives
			\[  \frac{1}{\sigma\sqrt{t}}\phi\left (\frac{y}{\sigma\sqrt{t}}\right ) = \frac{1}{2\pi} \int_{\mathbb{R}} e^{-i\theta y} \varphi_{t\sigma^2}(\theta ) \, d\theta . \quad \forall \, y \in \mathbb{R}. \]
			Similarly, the inverse formula for discrete variables (Exercise 3.3.2 in \cite{durret2019probability}) gives
			\[ \prob{Z_t=k} = \frac{1}{2\pi} \int_{-\pi}^{\pi} e^{-i\theta k} \varphi_Z(\theta ) \, d\theta  \quad \forall \, k \in \mathbb{Z}. \]
			Consequently, for any $k \in \mathbb{Z}$ we have
			\begin{equation}\label{eq: bound density with chf}
				\left |\prob{Z_t=k} -  \frac{1}{\sigma\sqrt{t}}\phi\left (\frac{k}{\sigma\sqrt{t}}\right )\right | \leq \frac{1}{2\pi} \left |\int_{-\pi}^{\pi} e^{-i\theta k} (\varphi_Z(\theta )-\varphi_{t\sigma^2}(\theta ))\, d\theta  \right | + \frac{1}{2\pi} \left |\int_{|\theta|> \pi} e^{-i\theta k} \varphi_{t\sigma^2}(\theta )\, d\theta  \right |.
			\end{equation}
			It is not difficult to bound the second term in the left hand side of (\ref{eq: bound density with chf}). Observe that
			\begin{align}\label{eq: main theorem bound 1}
				\left |\int_{|\theta |> \pi} e^{-i\theta k} \varphi_{t\sigma^2}(\theta )\, d\theta  \right | & \leq \int_{|\theta |>\pi} |\varphi_{t\sigma^2}(\theta )| \, d\theta  = 2\int_{\pi}^{\infty} e^{-t\sigma^2\theta^2/2}\, d\theta = 2 \frac{1}{\sigma\sqrt{t}} \int_{\pi\sigma \sqrt{t} }^\infty e^{-x^2/2} \, dx\\
				& \leq 2 \frac{1}{\sigma\sqrt{t}} \int_{\pi\sigma \sqrt{t} }^\infty \frac{x}{\pi\sigma\sqrt{t}}e^{-x^2/2} \, dx = \frac{2}{\pi t\sigma^2} e^{-\pi^2 t\sigma^2} \leq \frac{2}{\pi\sigma^2} \frac{1}{t}.\nonumber
			\end{align}
			To bound the remaining term, we break it into two parts using $\theta_0$. Write $c=(1-\lambda)^{-2}\|f\|_{\infty}^3(1+\|\frac{\pi-\mu}{\pi}\|_{2,\pi})$. For any $\theta  \in [-\theta_0,\theta_0]$, inequality (3.33) in \cite{Mann1996Berry} gives
			\[ \left | \varphi_Z(\theta)-e^{\frac{-t\sigma^2\theta^2}{2}}\right | \leq c e^{-t\sigma^2\theta^2/8}\left (683t|\theta|^3 + \frac{20}{\sigma^2}|\theta| \right ).
			\]
			Therefore, 
			\begin{align}\label{eq: main theorem bound 2}
				\left |\int_{-\theta_0}^{\theta_0} e^{-i\theta k} (\varphi_Z(\theta)-\varphi_{t\sigma^2}(\theta))\, d\theta \right |&\leq \int_{-\theta_0}^{\theta_0} |\varphi_Z(\theta)-\varphi_{t\sigma^2}(\theta)| \, d\theta \leq 2c \int_0^{\theta_0} e^{-t\sigma^2\theta^2/8}\left (683t\theta^3 + \frac{20}{\sigma^2}\theta \right ) \, d\theta \\ \nonumber
				&= 2c \frac{2}{\sigma \sqrt{t}} \int_0^{\theta_0 \sigma \sqrt{t}/2} e^{-x^2/2} \left (\frac{683 t \cdot 8}{\sigma^3 t^{3/2}} x^3 + \frac{20\cdot 2}{\sigma^3 \sqrt{t}} x\right ) \, dx\\ \nonumber
				&\leq \frac{4c}{\sigma^4 t} \left (5464 \int_0^\infty x^3e^{-x^2/2} \,dx + 40 \int_{0}^{\infty} xe^{-x^2/2}\, dx\right )\\ \nonumber
				&= \frac{43872c}{\sigma^4}  \frac{1}{t}. \nonumber
			\end{align}
			It remains to study the case $\theta_0<|\theta| \leq \pi$. Observe that
			\[\left |\int_{\theta_0\leq |\theta|\leq \pi} e^{-i\theta k} (\varphi_Z(\theta)-\varphi_{t\sigma^2}(\theta))\, d\theta \right | \leq \int_{\theta_0\leq |\theta|\leq \pi} |\varphi_Z(\theta)|\, d\theta + \int_{\theta_0\leq |\theta|\leq \pi} |\varphi_{t\sigma^2}(\theta)|\, d\theta.\]
			We know that $\varphi_{t\sigma^2}(\theta)= e^{-t\sigma^2 \theta^2/2}$, so we have
			\begin{equation}\label{eq: main theorem bound 3} \int_{\theta_0\leq |\theta|\leq \pi} |\varphi_{t\sigma^2}(\theta)|\, d\theta=2\int_{\theta_0}^{\pi} e^{-t \sigma^2 \theta^2/2} \, d\theta = \frac{2}{\sigma \sqrt{t}} \int_{\theta_0\sigma \sqrt{t}}^{\pi \sigma \sqrt{t}} e^{-\theta^2/2} \, d\theta \leq \frac{2}{\sigma\sqrt{t}} \frac{1}{\theta_0\sigma  \sqrt{t}} = \frac{2}{\theta_0\sigma^2 } \frac{1}{t}.
			\end{equation}

			Finally, recall that $Z_t=S_t+Y_t$ with $S_t = \sum_{j=1}^{b_t} V_j$, where $b_t \in \mathbb{N}$ and $V_j$ are independent and $\eta$-nonlattice for $\theta_0$. For $\theta$ with $\theta_0\leq |\theta|\leq \pi$ we have
			\begin{equation}\label{eq: bound decomposition} |\varphi_Z(\theta)|= |\varphi_S(\theta) \varphi_Y(\theta)|\leq|\varphi_S(\theta)|= \prod_{j=1}^{b_t}|\varphi_{V_j}(\theta)|=\prod_{j=1}^{b_t} |\mathbb{E}(e^{i\theta V_j})|\leq (1-\eta)^{b_t} \leq e^{-\eta b_t} \leq \frac{1}{e\eta b_t},
			\end{equation}
			where the last inequality holds since the function $f(x)=xe^{-\eta x}$ attains its maximum at $\eta^{-1}$. Therefore,
			\begin{equation}\label{eq: main theorem bound 4}\int_{\theta_0\leq |\theta|\leq \pi} |\varphi_{Z}(\theta)|\, d\theta \leq 2\int_{\theta_0}^{\pi} \frac{1}{e\eta b_t} \, d\theta \leq \frac{2\pi}{e \eta} \frac{1}{b_t}.
			\end{equation}
			In view of the bounds obtained in (\ref{eq: main theorem bound 1}), (\ref{eq: main theorem bound 2}), (\ref{eq: main theorem bound 3}), and (\ref{eq: main theorem bound 4}), we conclude that
			\[ \left |\prob{Z_t=k} -  \frac{1}{\sigma\sqrt{t}}\phi\left (\frac{k}{\sigma\sqrt{t}}\right )\right | \leq \left (\frac{1}{\pi^2\sigma^2 }  + \frac{21936c}{ \pi\sigma^4}  +\frac{1}{\pi\theta_0\sigma^2 }\right )\frac{1}{t} + \frac{1}{e\eta} \frac{1}{b_t}\quad  \forall \, k \in \mathbb{Z}.\]
			After substituting the value of $c$ and some straightforward computations, we obtain
			\[ \left  |\prob{Z_t=k} - \frac{1}{\sigma\sqrt{t}}\phi\left (\frac{k}{\sigma\sqrt{t}}\right )\right | \leq \left (\frac{1}{\theta_0 \sigma^2} + \frac{6983\|f\|_{\infty}^3}{\sigma^4(1-\lambda)^2}\left (1+\left \|\frac{\pi-\mu}{\pi}\right \|_{2,\pi}\right ) \right ) \frac{1}{t} + \frac{1}{e\eta} \frac{1}{b_t}.
			\]
			
			Consider now a general function $f\colon V \longrightarrow \mathbb{Z}$ and define $f_0=f-\mathbb{E}_\pi(f)$ and $Z_t^0=\sum_{i=0}^{t-1} f_0(X_i)$. Although $f_0$ might not be an integer-valued function, it takes values on $-\mathbb{E}_\pi(f)+\mathbb{Z}$. Also, $Z_t^0$ takes values on $-t\mathbb{E}_\pi(f) + \mathbb{Z}$ since
			\[ Z_t^0 = \sum_{i=0}^{t-1} f_0(X_i) = Z_t - t\mathbb{E}_\pi(f).\]
			In the previous argument, we only used that $Z_t$ takes integer values when applying the inverse formula, which is also valid for $-\mathbb{E}_\pi(f)-\mathbb{Z}$ (see Exercise 3.3.2 in \cite{durret2019probability}). Moreover, we can write $Z_t^0=S_t + Y_t^0$, where $Y_t^0= Y_t - t\mathbb{E}\pi(f)$. Therefore, the above argument applied to $Z_t^0$ gives 
			\[ \left  |\prob{Z_t^0=k-t\mathbb{E}_\pi(f)} - \frac{1}{\sigma\sqrt{t}}\phi\left (\frac{k-t\mathbb{E}_\pi(f)}{\sigma\sqrt{t}}\right )\right | \leq \left (\frac{1}{\theta_0 \sigma^2} + \frac{6983\|f\|_{\infty}^3}{\sigma^4(1-\lambda)^2}\left (1+\left \|\frac{\pi-\mu}{\pi}\right \|_{2,\pi}\right ) \right ) \frac{1}{t} + \frac{1}{e\eta} \frac{1}{b_t}.
			\]
			The result follows from $\prob{Z_t^0 = k-t\mathbb{E}_\pi(f)}=\prob{Z_t=k}$.
		\end{proof}
		
		\begin{proof}[Proof of Theorem \ref{theo: local CLT Markov}]
			Fix $t\ \in \mathbb{N}$, define $Z'_t=S_t+Y_t$, and denote its ch.f. by $\varphi_{Z'}$. Notice that the hypothesis $Z_t=S_t+Y_t$ in the proof of Lemma \ref{lem: Z_n=S_n+Y_n} was exclusively used to prove (\ref{eq: bound decomposition}), which is clearly true for $\varphi_{Z'}$. Therefore, it is enough to bound $|\varphi_Z(\theta) - \varphi_{Z'}(\theta)|$ for $\theta\in [-\pi,\pi]$. First, recall that by the mean value theorem we have $|e^{ix}-e^{iy}|\leq |x-y|$ for $x$, $y \in \mathbb{R}$. Consequently,
			\begin{align*}
				|\varphi_Z(\theta) - \varphi_{Z'}(\theta) | &= |\mathbb{E}_{\mu_0}(e^{i\theta Z_t}) - \mathbb{E}_{\mu_0}(e^{i\theta Z'_t})| \leq \mathbb{E}_{\mu_0}(|e^{i\theta Z_t}-e^{i\theta Z'_t}|) \leq |\theta| \mathbb{E}_{\mu_0}(|Z_t-Z'_t|) \leq \pi \mathbb{E}_{\mu_0}(|Z_t-Z'_t|),
			\end{align*}
			for any $\theta \in [-\pi,\pi]$. Thus, the hypothesis on $\mathbb{E}_{\mu_0}(|Z_t-Z'_t|)$ and (\ref{eq: bound decomposition}) give
			\begin{equation}\label{eq: bound decomposition 2} |\varphi_Z(\theta)| \leq |\varphi_{Z'}(\theta)| + |\varphi_Z(\theta) - \varphi_{Z'}(\theta)| \leq \frac{1}{e\eta b_t} + \frac{\pi M}{t} \quad \forall \, \theta_0\leq |\theta|\leq \pi.
			\end{equation}
			To conclude the proof we just need to repeat the argument in the proof of Lemma \ref{lem: Z_n=S_n+Y_n}, using (\ref{eq: bound decomposition 2}) instead of (\ref{eq: bound decomposition}).
		\end{proof}

		\section{Proof of Theorem \ref{theo:main}}\label{sec:main result}

		Let $G$ be a $d$-regular $\lambda$-expander graph. Let $(X_i)$ be the SRW on $G$ and write $Z_t=\sum_{i=0}^{t-1} \Val(X_i)$, where $\Val$ is a balanced labelling on $G$. We need some preliminary results to define the random variables $S_t$ and $Y_t$ that we use to decompose $Z_t$. The next lemma is a tighter version of the classic expander mixing lemma.
		
		\begin{lemma}[Lemma 4.15 in \cite{vadhan2009expander}]\label{lem:expander mixing lemma}
			Let $G$ be a $d$-regular $\lambda$-expander graph with $n$ vertices. For any subsets $F_1$, $F_2$ of $V$ we have
			\begin{equation}\label{eq: expander mixing lemma} 
				\left ||E(F_1,F_2)|-\frac{d}{n}|F_1||F_2|\right |\leq \lambda d \sqrt{\left (|F_1|-\frac{|F_1|^2}{n}\right )\left (|F_2|-\frac{|F_2|^2}{n}\right )},
			\end{equation}
			where $|E(F_1,F_2)|=\{(x,y) \in F_1 \times F_2 \colon \{x,y\} \in E\}|$ is the number of edges connecting $F_1$ and $F_2$ (counting edges contained in the intersection of $F_1$ and $F_2$ twice).
		\end{lemma}
		
		We will need lower bounds for $|E(F_1,F_2)|$ when $F_1$ is a subset of $V$ and $F_2$ is its complement.
		
		\begin{corollary}\label{cor:expander mixing lemma}
			Let $G$ be a $d$-regular $\lambda$-expander graph. For any $F_1\subseteq V$ and $F_2= \stcomp{F_1}$ we have
			\[  |E(F_1,F_2)|\geq \frac{1}{2}(1-\lambda)d \min\{|F_1|,|F_2|\}.\]
		\end{corollary}
		\begin{proof}
			Write $n=|V|$. Since $F_2=\stcomp{F_1}$, we have $|F_1|-\frac{|F_1|^2}{n}=|F_2|-\frac{|F_2|^2}{n}=\frac{|F_1||F_2|}{n}$. Thus, the right hand side in (\ref{eq: expander mixing lemma}) is $\lambda\frac{d}{n}|F_1||F_2|$. The result follows from the fact that $\frac{x(x-n)}{n}\geq \frac{x}{2}$ for every $x\in\left [0,\frac{n}{2}\right ]$.
		\end{proof}
	
		Given a labelling $\operatorname{val}\colon V \longrightarrow \{0,1\}$ on $G$, let $A=\{x \in V \colon \Val(x)=0\}$ and $B=\{x \in V \colon \Val(x)=1\}$. The labelling $\Val$ is balanced when $|A|=|B|=\frac{|V|}{2}$. Given $x \in V$, write $q(x)$ for the number of neighbors $y$ of $x$ with $\Val(y)=0$. Define the sets
		\[ A_j = \{x\in A\colon q(x)=j\} \quad \mbox{and} \quad B_j=\{x \in B \colon q(x)=j\} \quad \forall \, j \in \{0,\ldots,d\}.\]
		The next lemma shows that for some $k^* \in \{1,\ldots,d-1\}$ either the set $A_{k^*}$ or $B_{k^*}$ is relatively large.
		\begin{lemma}\label{lem:A_j or B_j}
			Let $G$ be a $d$-regular $\lambda$-expander graph and consider a balanced labelling $\Val$ on $G$. Write $\delta=\frac{(1-\lambda)^2}{3}$. Then there is $k^* \in \{1,\ldots,d-1\}$ such that either $$|A_{k^*}|\geq \frac{\delta|A|}{d-1}\quad \mbox{or} \quad |B_{k^*}|\geq \frac{\delta|B|}{d-1}.$$
		\end{lemma}
		
		\begin{proof}
			Assume that the statement is false. Then we must have
			\begin{equation}\label{eq:|A_0|+|A_d|} 
				|A_0|+|A_d|>(1-\delta)|A| \quad \mbox{and} \quad |B_0|+|B_d|> (1-\delta)|B|.
			\end{equation}
			Write $n=|V|$. Take $F_1=A$ and $F_2=B$. Corollary \ref{cor:expander mixing lemma} gives $|E(A,B)|\geq (1-\lambda) d\frac{n}{4}.$ Notice also that 
			$$|E(A,B)|\leq d |A\setminus A_d|\leq d(|A_0|+\delta|A|).$$
			Therefore, we obtain
			\begin{equation}\label{eq:Ad} 
				|A_0|\geq (1-\lambda)\frac{n}{4}-\delta\frac{n}{2}=(1-\lambda-2\delta)\frac{n}{4}.
			\end{equation}
			A completely analogous argument replacing $A_d$ with $B_0$ and $A_0$ with $B_d$ gives
			\begin{equation}\label{eq:B0} 
				|B_d|\geq (1-\lambda-2\delta)\frac{n}{4}.
			\end{equation}
			Next, take $F_1=F_2=A$. The expander mixing lemma gives $|E(A,A)|\geq (1-\lambda) d\frac{n}{4}.$ Observe also that
			\[ |E(A,A)|\leq d|A\setminus A_0|\leq d(|A_d|+\delta|A|).\]
			Thus, we get
			\begin{equation}\label{eq:A0} |A_d|\geq (1-\lambda-2\delta)\frac{n}{4}.
			\end{equation}
			Similarly, taking $F_1=F_2=B$ we obtain
			\begin{equation}\label{eq:Bd} 
				|B_0|\geq (1-\lambda-2\delta)\frac{n}{4}.
			\end{equation}
			Finally, take $F_1=A_d\cup B_0$ and $F_2=\stcomp{F_1}$. In view of (\ref{eq:Ad}), (\ref{eq:B0}), (\ref{eq:A0}), and (\ref{eq:Bd}) we have $\min\{|F_1|,|F_2|\}\geq (1-\lambda-2\delta)\frac{n}{2}$. Hence, Corollary \ref{cor:expander mixing lemma} gives
			\begin{align*} |E(F_1,F_2)|&\geq (1-\lambda) d (1-\lambda-2\delta)\frac{n}{4}=(1-\lambda)^2 d\left (1- \frac{2(1-\lambda)}{3}\right )\frac{n}{4}\geq d\delta\frac{n}{4}=2d\delta|A|.
			\end{align*}
			Therefore, we must have either $|E(A_d,F_2)|\geq d\delta |A|$ or $|E(B_0,F_2)|\geq d\delta |B|$. In the first case, for any $e=\{x,y\} \in E$ with $x \in A_d$ and $y \in F_2$, we must have $\Val(y)=0$. Thus, $y \in A\setminus A_d$. Moreover, $y$ cannot belong to $A_0$ since it is adjacent to $x$ and $\Val(x)=0$. Therefore, $y \in A\setminus(A_0\cup A_d)$. Consequently,
			\[ |A\setminus(A_0\cup A_d)|\geq \frac{|E(A_0,F_2)|}{d}\geq \delta|A|,\]
			which contradicts (\ref{eq:|A_0|+|A_d|}). If $|E(B_0,F_2)|\geq d \delta |B|$ we obtain $|B\setminus(B_0\cup B_d)|\geq \delta|B|$, also a contradiction.\qedhere
		\end{proof}
	
		Let $(X_i)$ be the SRW on a $d$-regular $\lambda$-expander graph $G$ with uniform initial distribution $\pi$. Recall that $Z_t=\sum_{i=0}^{t-1} \Val(X_i)$ for some labelling $\Val$. The sets $A_{k^*}$ and $B_{k^*}$ provided by Lemma \ref{lem:A_j or B_j} will be used to decompose $Z_t$ as a sum of two convenient independent random variables $S_t$ and $Y_t$. The idea is that as we run the chain, we frequently see cycles of length $2$. More precisely, $X_i=X_{i+2}$ with probability $\frac{1}{d}$. If we have $|A_{k^*}|\geq \frac{\delta|A|}{d-1}$, then many of these $2$-cycles will start at a vertex of $A_{k^*}$. The contribution to $Z_t$ of each one of these $2$-cycles is either $0$ with probability $\frac{k^*}{d}$, or $1$ with probability $\frac{d-k^*}{d}$, and these contributions are independent of each other. If $S_t$ represents the total contribution of the $2$-cycles starting from $A_{k^*}$, and $Y_t$ represents the contribution of the rest of the walk, then $S_t$ is a sum of i.i.d. Bernoulli random variables and $Z_t=S_t+Y_t$. Although the idea is simple, making it rigorous requires a careful analysis.

		Apply Lemma \ref{lem:A_j or B_j} to find $k^* \in \{1,\ldots,d-1\}$ such that $\max\{|A_{k^*}|,|B_{k^*}|\}\geq \frac{\delta|A|}{d-1}$, where $\delta=\frac{(1-\lambda)^2}{3}$. By symmetry, we may assume that $|A_{k^*}|\geq \frac{\delta|A|}{d-1}$. Let $(X_i^2)$ be the $2$-steps SRW, that is, the Markov chain with transition matrix $P^2$. Let $N_t$ be the number of times that $X_i^2 \in A_{k^*}$ within the first $\lfloor t/2 \rfloor-1$  steps of the chain. Then
		\begin{equation}\label{eq: bound Nt} \mathbb{E}_\pi(N_t)= \lfloor t/2\rfloor \pi(A_{k^*})= \lfloor t/2\rfloor \frac{|A_{k^*}|}{|V|} = \lfloor t/2\rfloor \frac{|A_{k^*}|}{2|A|} \geq \lfloor t/2\rfloor \frac{\delta}{2(d-1)} \geq \frac{\delta(t-2)}{4(d-1)}=\frac{(1-\lambda)^2}{12(d-1)}(t-2).
		\end{equation}

		Recall that $\prob{X_i=X_{i+2}}=\frac{1}{d}$ for any $i\geq 0$. Let $\widetilde{N}_{t}$ be a random variable that counts the number of times that one of these cycles of length $2$ starting from a vertex of $A_{k^*}$ appears at even time within the first $t$ steps. To be more precise, let $i_1,\ldots, i_{N_t}$ be the times for which $X_i^2 \in A_{k^*}$ and let $U_j$ be the indicator that $X_{i_j}=X_{i_j+2}$, that is,
		\[ U_j=\begin{cases}
			1 & \mbox{ if } X_{i_j}=X_{i_j+2};\\
			0 & \mbox{ otherwise}.
		\end{cases}\]
		Then $(U_j)$ is a sequence of independent Bernoulli random variables with parameter $\frac{1}{d}$. We define 
		\begin{equation}\label{eq: tilde(N)_t} 
			\widetilde{N}_t = \sum_{j=1}^{N_t} U_j \quad \mbox{ and } \quad B_t=\sum_{j=1}^{\lceil \frac{\mathbb{E}_\pi(N_t)}{2} \rceil} U_j.
		\end{equation}
		
		Let $b_t= \lfloor\mathbb{E}_\pi(\widetilde{N}_t)/4\rfloor$. For every $i \in\{1,\ldots,\widetilde{N}_t\}$, write $(x_i,y_i)$ for the vertices appearing in the $i$-th $2$-cycle, where $x_i\in A_{k^*}$ and $y_i$ is some neighbor of $x_i$. Let $V_i$ be the indicator of the event that $\Val(y_i)=1$ (which happens with probability $\frac{d-k^*}{d}$). Then $V_1,V_2,\ldots$ are i.i.d. Bernoulli random variables. Moreover, the $i$-th $2$-cycle adds $V_i$ to the total sum of the labels. Consider as well $\widetilde{V}_1,\widetilde{V}_2,\ldots$ independent from all previous random variables and identically distributed random variables given by
		\[
		\widetilde{V}_i= 
		\begin{cases}
			1 &\quad\mbox{with probability } \frac{d-k^*}{d};\\
			0 &\quad \mbox{ otherwise.} \\ 
		\end{cases}
		\]
		We can finally introduce the random variables used to decompose $Z_t$. Define
		\begin{equation}\label{eq: S_t and Y_t definition} S'_t=\sum_{i=1}^{\min\{b_t,\widetilde{N}_t\}} V_i, \quad Y_t= Z_t - S'_t, \quad \mbox{and}\quad S_t=S'_t+\sum_{i=1}^{(b_t-\widetilde{N}_t)^+} \widetilde{V}_i,
		\end{equation}
		where $(b_t-\widetilde{N}_t)^+=\max \{b_t-\widetilde{N}_t , 0\}$. It turns out that $\widetilde{N}_t$ is concentrated around its mean, so one should expect $b_t\leq \widetilde{N}_t$. Hence, $S'_t$ is equal to $S_t$ with high probability, or equivalently, $Z_t=S_t+Y_t$ with high probability. To prove this fact, first we need the following Chernoff-type tail bound for the binomial distribution. For $a>0$ set $\varphi(a)=1-a+a\log a$. Then $\varphi(a)>0$ for $a\neq 1$ and $\varphi(1)=0$.
		
		\begin{lemma}[Lemma 8.1 in \cite{penrose2011local}]\label{lem: binomial}
			Let $N$ be a binomial distributed random variable with $\mathbb{E}(N)=\mu>0$. Then
			\[ \prob{N\leq x}\leq e^{-\mu\varphi\left (\frac{x}{\mu}\right )}\quad \forall \, 0<x\leq \mu.\]
		\end{lemma}
	
		We also need the following consequence of Theorem 2.1 in \cite{gillman1998a}. It provides a Chernoff bound for random walks on expander graphs.
		
		\begin{lemma}\label{lem:chernoff}
			Let $(X_i)$ be the random walk on a weighted graph $G=(V,E)$ starting from stationary distribution $\pi$, and let $1-\lambda^*$ be its absolute spectral gap. Given $A \subseteq V$, let $N_t$ be the number of visits to $A$ in $t$ steps. For any $0<\gamma\leq t$,
			\[ \prob{|N_t-\mathbb{E}_\pi(N_t)| \geq \gamma} \leq 4e^{-\gamma^2(1-\lambda^*)/20t}. \]
		\end{lemma}

		Now we can show that $\widetilde{N}_t$ is expected to be bigger than $b_t$.
		
		\begin{lemma}\label{lem: tilde Nt}
			Let $\widetilde{N}_t$ be defined as in (\ref{eq: tilde(N)_t}). Then
			\[ \prob{\widetilde{N}_t \leq \mathbb{E}_\pi(\widetilde{N}_t)/4} \leq 5\operatorname{exp}\left (-\frac{(1-\lambda)^5}{11520d^2}(t-4)\right ). \]
		\end{lemma}
		
		\begin{proof}
			Notice that
			
			\begin{align*} \prob{\widetilde{N}_t \leq \mathbb{E}_\pi(\tilde{N_t})/4} &\leq \prob{\widetilde{N}_t \leq \mathbb{E}_\pi(\tilde{N_t})/4 | N_t> \mathbb{E}_\pi(N_t)/2} + \prob{N_t \leq \mathbb{E}_\pi(N_t)/2}\\
				& \leq \prob{B_t\leq \mathbb{E}_\pi(\tilde{N_t})/4} + \prob{N_t \leq \mathbb{E}_\pi(N_t)/2}\\
				& \leq \prob{B_t\leq \mathbb{E}_\pi(B_t)/2} + \prob{N_t \leq \mathbb{E}_\pi(N_t)/2},
			\end{align*}
			where the last inequality follows from the fact that $\frac{\mathbb{E}_\pi(\widetilde{N}_t)}{4} \leq \frac{\mathbb{E}_\pi(B_t)}{2}$. We can use Lemma \ref{lem: binomial} to get
			\[ \prob{B_t\leq \mathbb{E}_\pi(B_t)/2} \leq e^{-\mathbb{E}_\pi(B_t) \varphi(1/2)} \leq \operatorname{exp}\left (-\frac{\delta\varphi(1/2)}{8d(d-1)}(t-2)\right ) \leq \operatorname{exp}\left (-\frac{(1-\lambda)^2}{157d^2}(t-2)\right ), \] 
			where $\delta=(1-\lambda)^2/3$. Finally, applying Lemma \ref{lem:chernoff} to $(X_i^2)$ and $A_{k^*}$ with $\gamma=\mathbb{E}_\pi(N_t)/2$ gives 
			\[ \prob{N_t \leq \mathbb{E}_\pi(N_t)/2}\leq 4e^{-\mathbb{E}_\pi(N_t)^2(1-\lambda^2)/80t}\leq 4\operatorname{exp}\left (- \frac{(1-\lambda^2)\delta^2}{16(d-1)^2} \frac{(t-2)^2}{80t}\right )\leq 4\operatorname{exp}\left (-\frac{(1-\lambda)^5}{11520d^2}(t-4)\right ).
			\]
		\end{proof}
		
		We can use Lemma \ref{lem: tilde Nt} to obtain the bound for $\mathbb{E}_\pi(|Z_t-S_t-Y_t|)$ that we need to apply Theorem \ref{theo: local CLT Markov}.
		
		\begin{lemma}\label{lem: high probability}
			Let $S_t$, $S'_t$, and $Y_t$ defined as in (\ref{eq: S_t and Y_t definition}). Then
			\[\mathbb{E}_\pi(|Z_t-S_t-Y_t|) \leq  \frac{10^{10} d^3}{(1-\lambda)^4} \frac{1}{t}.\]
		\end{lemma}
		
		\begin{proof}
			Note that the function $f(t)=t^2e^{-\alpha t}$ attains its maximum at $t=2\alpha^{-1}$. Write $\alpha=\frac{(1-\lambda)^2}{11520d^2}$. Then Lemma \ref{lem: tilde Nt} gives
			\begin{align*}
				\prob{\widetilde{N}_t \leq \mathbb{E}_\pi(\widetilde{N}_t)/4} &\leq 5e^{-\alpha (t-4)}= \left (5e^4t^2e^{-\alpha t} \right )\frac{1}{t^2}\leq \left (5e^4 \frac{4e^{-2}}{\alpha^2}\right )\frac{1}{t^2}\leq \frac{2\cdot 10^{10} d^4}{(1-\lambda)^4} \frac{1}{t^2}.
			\end{align*}
			Observe that $|Z_t-S_t-Y_t| = |S'_t - S_t| \leq b_t \leq \frac{t}{4d}$. Since $Z_t=Y_t+S_t$ if $\widetilde{N}_t\geq b_t$, we have
			\[ \mathbb{E}_\pi(|Z_t-S_t-Y_t|) \leq \frac{t}{4d} \prob{Z_t \neq S_t + Y_t} \leq  \frac{10^{10} d^3}{(1-\lambda)^4} \frac{1}{t}.\qedhere\]

		\end{proof}
		Finally, the next simple result shows that a Bernoulli random variable is nonlattice. We include its proof for completeness.
		
		\begin{lemma}\label{lem: nonlattice}
			Let $V$ be a Bernoulli random variable with parameter $p \in (0,1)$ and take $\theta_0 \in (0,\pi)$. Then $X$ is a $\eta$-nonlattice variable for $\theta_0$ with 
			\[ \eta= p(1-p)(1-\cos(\theta_0)).\]
		\end{lemma}
		
		\begin{proof}
			Write $\varphi_V$ for the characteristic function of $V$, that is,
			\[\varphi_V(\theta)= \mathbb{E}(e^{i\theta V})= (1-p) + pe^{i\theta} \quad \forall \, \theta \in [-\pi,\pi],\]
			which is a convex combination of the points $1$ and $e^{i\theta}$. Therefore, when $\theta_0\leq |\theta|\leq \pi$ we clearly have $|\varphi_V(\theta)|\leq |\varphi_V(\theta_0)|$. A simple computation yields
			\[ |\varphi_V(\theta_0)|^2= p^2\sin(\theta_0)^2 + ((1-p)+p\cos(\theta_0))^2= p^2+(1-p)^2 +2p(1-p)\cos(\theta_0).\]
			Using the equality $(a-b)(a+b)=a^2-b^2$ and $|\varphi_V|\leq 1$ we conclude that
			\[ 1-|\varphi_V(\theta_0)| \geq \frac{1}{2} (1-|\varphi_V(\theta_0)|^2)=p(1-p)(1-\cos(\theta_0)).\qedhere\]
		\end{proof}
	
		We can now present the proof of our main result Theorem \ref{theo:main}. Although it does not optimize the constant, it shows that we can take
		\[ C_1(\lambda,d)= \frac{2\cdot 10^{13}d^9}{(1-\lambda)^{10}}.\]

		\begin{proof}[Proof of Theorem \ref{theo:main}]
			Consider the random variables $\widetilde{N}_t$, $V_i$, $\widetilde{V}_i$, $S'_t$, $S_t$, and $Y_t$ appearing in (\ref{eq: S_t and Y_t definition}). Let us check that the hypotheses of Theorem \ref{theo: local CLT Markov} are satisfied. First, $(X_i)$ is the SRW on a finite $\lambda$-expander graph with $\lambda<1$, whence it is irreducible and aperiodic. The function considered is $\Val\colon V \longrightarrow \mathbb{Z}$. Observe that $S_t$ does not get affected by conditioning on $Y_t$. Indeed, regardless of the value of $Y_t$, $S_t$ is a sum of $b_t$ i.i.d. Bernouilli random variables that are independent of $Y_t$. Therefore, $S_t$ and $Y_t$ are independent. Moreover, Lemma \ref{lem: high probability} gives $\mathbb{E}_\pi(|Z_t-S_t-Y_t|)\leq M/t$, with 
			\[ M=\frac{10^{10} d^3}{(1-\lambda)^4}. \] 
			Finally, Lemma \ref{lem: nonlattice} shows that $V_i$ and $\widetilde{V}_i$ are $\eta$-nonlattice for $\theta_0=(1-\lambda)^2\sigma^2/2708$ with 
			\[ \eta\geq \frac{1}{d} \left (1-\frac{1}{d}\right )(1-\cos(\theta_0)) = \frac{d-1}{d^2}(1-\cos(\theta_0)).\]
			Since in our case $\|\Val\|_{\infty}=1$ and $\mu=\pi$, Theorem \ref{theo: local CLT Markov} gives
			\begin{equation}\label{eq: theorem 4.1 applied} \left  |\prob{Z_t=k} - \frac{1}{\sigma\sqrt{t}}\phi\left (\frac{k-t/2}{\sigma\sqrt{t}}\right )\right | \leq \left (\pi M +\frac{1}{\theta_0 \sigma^2} + \frac{C_4}{\sigma^4(1-\lambda)^2} \right ) \frac{1}{t} + \frac{1}{e\eta} \frac{1}{b_t}.
			\end{equation}
			Recall that $b_t=\lfloor \mathbb{E}_\pi(\widetilde{N}_t)/4 \rfloor$. Using the bound (\ref{eq: bound Nt}) for $\mathbb{E}_\pi(N_t)$ we get
			\[
				\frac{\mathbb{E}_\pi(\widetilde{N}_t)}{4}= \frac{\mathbb{E}_\pi(N_t)}{4d} \geq \frac{(1-\lambda)^2}{48d(d-1)}(t-2)\geq \frac{(1-\lambda)^2t-2}{48d(d-1)}.
			\]
			If $t \leq 48(1-\lambda)^{-2}d^3$ there is nothing to prove. Otherwise, a simple computation shows that
			\begin{equation}\label{eq: b_t}
				b_t=\left \lfloor \frac{\mathbb{E}_\pi(\widetilde{N}_t)}{4}\right  \rfloor \geq \frac{\mathbb{E}_\pi(\widetilde{N}_t)}{4} -1 \geq \frac{(1-\lambda)^2t -48d^2}{48d(d-1)}\geq \frac{(1-\lambda)^2}{48d^2}t.
			\end{equation}
			Next, we provide a bound for $\sigma$. Observe that
			\begin{align*}
				\Var(Z_t)&\geq \mathbb{E}_\pi(\Var(Z_t|\widetilde{N}_t)) \geq \sum_{u= b_t }^\infty \Var(Z_t|\widetilde{N}_t=u) \prob{\widetilde{N}_t=u}=\sum_{u= b_t }^\infty \Var(S_t + Y_t|\widetilde{N}_t=u) \prob{\widetilde{N}_t=u}\\
				&\geq \sum_{u= b_t }^\infty \Var(S_t|\widetilde{N}_t=u) \prob{\widetilde{N}_t=u}= \Var(S_t) \sum_{u= b_t }^\infty \prob{\widetilde{N}_t=u} \geq b_t\frac{d-1}{d^2} \prob{\widetilde{N}_t\geq b_t}.
			\end{align*}
			As $t$ grows, Lemma \ref{lem: tilde Nt} shows that $\prob{\widetilde{N}_t\geq b_t}$ tends to $1$. Thus, (\ref{eq: b_t}) gives
			\begin{equation}\label{eq: bound sigma} 
				\sigma \geq \sqrt{\frac{(1-\lambda)^2(d-1)}{48d^4}} \geq \sqrt{\frac{(1-\lambda)^2}{96d^3}} \geq \frac{(1-\lambda)}{10d^{3/2}}.
			\end{equation}
			Finally, we give a bound for $\eta$. It is straightforward to show that $\cos(x)\leq 1-x^2/5$ for any $x \in [-\pi,\pi]$. Consequently,
			\begin{equation}\label{eq: bound eta}\eta \geq \frac{d-1}{d^2}(1-\cos(\theta_0)) \geq \frac{1}{2d} \frac{\theta_0^2}{5}= \frac{(1-\lambda)^4\sigma^4}{10\cdot 2708^2d} \geq \frac{(1-\lambda)^8}{8\cdot 10^{11}d^7}. 
			\end{equation}
		Substituting the value of $M$ and $\theta_0$ and applying the bounds (\ref{eq: b_t}), (\ref{eq: bound sigma}), and (\ref{eq: bound eta}) in (\ref{eq: theorem 4.1 applied}) yields
		\[\left  |\prob{Z_t=k} - \frac{1}{\sigma\sqrt{t}}\phi\left (\frac{k-t/2}{\sigma\sqrt{t}}\right )\right | \leq \left (\frac{\pi10^{10}d^3}{(1-\lambda)^4} + \frac{10^8 d^6}{(1-\lambda)^6} + \frac{1.5 \cdot 10^{13} d^9}{(1-\lambda)^{10}}\right ) \frac{1}{t} \leq  \frac{2\cdot 10^{13}d^9}{(1-\lambda)^{10}} \frac{1}{t}.\qedhere
		\]
		\end{proof}
	
		\begin{remark}\label{remark better bound}
			It is possible to obtain a better bound in Theorem \ref{theo:main} if we assume that $\lambda$ is small enough. If $\lambda\leq \frac{1}{5}$, then Theorem \ref{theo:main} holds with $C_1(d)=4\cdot10^{11}d^3$. This will be proved in Section \ref{sec:variance}.
		\end{remark}
		
		\section{Proofs of Proposition \ref{prop: variance} and Corollary \ref{cor:main}}\label{sec:variance}
		Let $G=(V,E)$ be a $d$-regular $\lambda$-expander graph, let $(X_i)$ be the SRW on $G$ with uniform initial distribution $\pi$, and let $\Val \colon V\longrightarrow \{0,1\}$ be a labelling with $\alpha=\mathbb{E}_\pi(\Val)\in (0,1)$. For simplicity, write $Y_i=\Val(X_i)$. Recall that $Z_t=\sum_{i=0}^{t-1} Y_i$, thus its variance can be expressed as
		\begin{equation}\label{eq:var of sum} 
			\Var(Z_t)=\sum_{j=0}^{t-1} \Var(Y_j) + 2\sum_{i<j} \Cov(Y_i,Y_j).
		\end{equation}
		It is clear that $\Var(Y_j)=\mathbb{E}_\pi(Y_j^2)- \mathbb{E}_\pi(Y_j)^2 = \alpha - \alpha^2=\alpha(1-\alpha)$ and $\Cov(Y_i,Y_j)=\mathbb{E}_\pi(Y_iY_j) - \alpha^2$. Recall that $A=\{x \in V \colon \Val(x)=0\}$ and $B=\stcomp{A}$. We have
		\begin{equation}\label{eq: E(Y_i Y_j)} \mathbb{E}_\pi(Y_iY_j)= \prob{Y_i=1,Y_j=1}=\alpha \prob{Y_j=1|Y_i=1}=\alpha \prob{X_{j-i}\in B|X_0 \in B}.
		\end{equation}
		Thus, we want to find the probability that the chain is at a vertex of $B$ after $j-i$ steps when the initial vertex is chosen uniformly at random from $B$.
		
		\begin{lemma}\label{lem:prob Xk in A}
			Let $G=(V,E)$ be a $d$-regular graph with $n$ vertices and let $B\subseteq V$. If $(X_i)$ is the simple random walk on $V$ starting uniformly at random from $B$, we have
			\[\prob{X_k \in B} = \pi(B)+ \pi(B)\sum_{j=2}^n \langle \pi_B,f_j\rangle^2 \lambda^k_j \quad \forall \, k \in \mathbb{N},\]
			where $(f_j)_{j=1}^n$ is the orthonormal basis of eigenvectors corresponding to the eigenvalues $(\lambda_j)_{j=1}^n$ and $\pi_B$ is the uniform distribution on $B$.
		\end{lemma}
		
		\begin{proof}
			Let $P$ denote the transition matrix of $(X_i)$. Since $P$ is symmetric, $P^k \pi_B$ is the vector of probabilities of the chain $(X_i)$ after $k$ steps. Therefore, $\prob{X_k \in B}=\langle P^k \pi_B,\Val\rangle$. We can use (\ref{eq:spectral form Pf}) to decompose $P$ and obtain
			\[ P^k \pi_B= \sum_{j=1}^n \langle \pi_B,f_j\rangle_\pi f_j \lambda^k_j=\pi + \sum_{j=2}^n \langle \pi_B,f_j\rangle_\pi f_j \lambda^k_j.\]
			Since $\pi_B=\frac{1}{|B|}\Val$, we conclude
			\[ \langle P^k\pi_B,\Val\rangle = \pi(B) + \sum_{j=2}^n \langle \pi_B,f_j\rangle_\pi \langle f_j,\Val\rangle \lambda^k_j= \pi(B)+\pi(B) \sum_{j=2}^n \langle \pi_B,f_j\rangle^2 \lambda^k_j.\qedhere\]
		\end{proof}
		
		\begin{proof}[Proof of Proposition \ref{prop: variance}]
		Recall that $A=\{x\in V \colon \Val(x)=0\}$, $B=\stcomp{A}$, and $\alpha=\pi(B)$. Fix $0\leq i<j\leq t-1$. Lemma \ref{lem:prob Xk in A} and (\ref{eq: E(Y_i Y_j)}) give
		\[ \Cov(Y_i,Y_j)=\mathbb{E}_\pi(Y_i Y_j)- \alpha^2 = \alpha^2 \sum_{k=2}^n \langle \pi_B,f_j\rangle^2\lambda^{j-i}_k \quad \forall\, i<j.\]
		Adding all covariances yields
		\[ \sum_{i<j} \Cov (Y_i,Y_j) = \alpha^2\sum_{k=1}^{t-1}(t-k)\sum_{j=2}^n\langle \pi_B,f_j\rangle^2\lambda^{k}_j.\]
		The formula (\ref{eq:var formula}) follows now from (\ref{eq:var of sum}). Finally, if we asssume that $G$ is a $\lambda$-expander we get
		\begin{align*}
			\left |\Var(Z_t)-\alpha(1-\alpha)t\right |&\leq 2\alpha^2 \sum_{k=1}^{t-1}(t-k)\lambda^{k}\sum_{j=2}^n\langle \pi_B,f_j\rangle^2\leq 2\alpha^2 t \sum_{k=1}^{t-1} \lambda^k (n\|\pi_B\|_2^2-\langle \pi_B,f_1\rangle^2)\\
			&= 2\alpha^2 \left (\frac{n}{|B|}-1\right ) t\sum_{k=1}^{t-1}\lambda^k 
			\leq 2\alpha(1-\alpha)t\frac{\lambda}{1-\lambda}.\qedhere
		\end{align*}
		\end{proof}
	
		As Remark \ref{remark better bound} claims, for small values of $\lambda$ we can use the bound for the variance provided by Proposition \ref{prop: variance} to obtain a better bound in Theorem \ref{theo:main}.
		
		\begin{proof}[Proof of Remark \ref{remark better bound}]
			In this case, we have $\alpha=1/2$. Dividing by $t$ and sending $t$ to infinity in (\ref{eq:var bound}) give
			\[ \left |\sigma^2 - \frac{1}{4} \right | \leq \frac{1}{2} \frac{\lambda}{1-\lambda}\leq \frac{1}{8}, \]
			since we assume $\lambda\leq 1/5$. Therefore, $\sigma^2\geq 1/8$. Using this bound instead of (\ref{eq: bound sigma}) also allow us to obtain better bounds for $\theta_0$ and $\eta$. Indeed, we have
			$$\theta_0 = \frac{(1-\lambda)^2\sigma^2}{2708} \geq \frac{1}{33850} \quad \mbox{and} \quad \eta \geq \frac{1}{2d} \frac{\theta_0^2}{5} \geq \frac{1}{1.2 \cdot 10^{10}d}.$$
			Using these bounds and the bound (\ref{eq: b_t}) for $b_t$ in (\ref{eq: theorem 4.1 applied}) yields
			\[ \left  |\prob{Z_t=k} - \frac{1}{\sigma\sqrt{t}}\phi\left (\frac{k-t/2}{\sigma\sqrt{t}}\right )\right |  \leq \left (8\cdot 10^{10}d^3 + 10^6 + 3.4 \cdot 10^{11}d^3\right ) \frac{1}{t} \leq  4\cdot 10^{11}d^3 \frac{1}{t}.\qedhere
			\]
			
		\end{proof}

		Finally, we prove Corollary \ref{cor:main}. First, we show that the normalizing constant $D(\mu,\sigma^2)$ appearing in (\ref{eq:disc normal}) is close to $1$ when the variance is large.
		
		\begin{lemma}\label{lem: bound normalizing constant}
			Let $D(\mu,\sigma^2)$ be the normalizing constant appearing in (\ref{eq:disc normal}). For $\sigma^2\geq1$ we have 
			$$|1-D(\mu,\sigma^2)| \leq \frac{1}{\sqrt{2\pi}} \frac{1}{\sigma}.$$
		\end{lemma}

		\begin{proof}
			The density function of a normal distribution increases on $(-\infty, 0)$ and decreases on $(0,\infty)$. Therefore,
			\[\phi(k)\geq \sigma\int_{k-\frac{1}{\sigma}}^k \phi(x)\, dx \quad \mbox{ for } k<0 \quad \mbox{ and } \quad \phi(k)\geq \sigma \int_{k}^{k+\frac{1}{\sigma}} \phi(x) \, dx \quad \mbox{ for } k \geq 0.\] 
			Let $I_1=\{ k \in \mathbb{Z} \colon k \leq \mu\}$ and $I_2= \mathbb{Z} \setminus I_1$. The previous observation gives
			\[ \sum_{k \in I_1} \sigma^{-1}\phi\left (\frac{k-\mu}{\sigma}\right ) \geq \sum_{k \in I_1}  \int_{ \frac{k-\mu}{\sigma} - \frac{1}{\sigma}}^{\frac{k-\mu}{\sigma}} \phi(x) \, dx = \int_{-\infty}^{\frac{\lfloor \mu \rfloor -\mu}{\sigma}} \phi(x) \, dx.\]
			Similarly,
			\[ \sum_{k \in I_2} \sigma^{-1}\phi\left (\frac{k-\mu}{\sigma}\right ) \geq \sum_{k \in I_2}  \int_{ \frac{k-\mu}{\sigma}}^{\frac{k-\mu}{\sigma} + \frac{1}{\sigma}} \phi(x) \, dx = \int_{\frac{\lfloor \mu \rfloor -\mu}{\sigma} + \frac{1}{\sigma}}^{\infty} \phi(x) \, dx.\]
			Since $\phi$ attains its maximum at $x=0$, we deduce that
			\[ \sum_{k \in\mathbb{Z}} \sigma^{-1}\phi\left (\frac{k-\mu}{\sigma}\right ) \geq 1- \frac{1}{\sigma} \phi(0) =1-\frac{1}{\sqrt{2\pi}}\frac{1}{\sigma}.\]
			Similarly, $\phi(k)\leq \sigma\int_{k}^{k+1/\sigma} \phi(x) \, dx$ for any $k\leq -\frac{1}{2\sigma}$ and $\phi(k)\leq \sigma \int_{k-1/\sigma}^{k} \phi(x)\, dx $ for $k\geq \frac{1}{2\sigma}$. Notice that there is a unique integer $k^* \in \mathbb{Z}$ such that 
			\[ -\frac{1}{2\sigma} \leq \frac{k^*-\mu}{\sigma} < \frac{1}{2\sigma}.\] Therefore,
			\[ \sum_{k \in \mathbb{Z}} \sigma^{-1} \phi\left (\frac{k-\mu}{\sigma} \right ) \leq \sum_{k \in \mathbb{Z}\setminus\{k^*\}} \sigma^{-1} \phi\left (\frac{k-\mu}{\sigma} \right ) + \frac{1}{\sigma}\phi(0) \leq 1+\frac{1}{\sigma} \phi(0) = 1+\frac{1}{\sqrt{2\pi}} \frac{1}{\sigma}.\qedhere\]
		\end{proof}
		
		Let $\Val$ be a balanced labelling on $G$ and consider $Z_t=\sum_{i=0}^{t-1} \Val(X_i)$. We write $\varphi_{Z}$ for the ch.f. of $Z_t$ and $\varphi_{t}$ for the ch.f. of a normal distribution with mean $t/2$ and variance $t\sigma^2$, where $\sigma^2 = \lim_{t \to \infty} \Var(Z_t)/t$. The next technical result bounds the $\operatorname{L}^2$-distance between $\varphi_Z$ and $\varphi_t$.
		\begin{lemma}\label{lem: L2 norm of varphiZ and varphit}
			The characteristic function of $Z_t$ satisfies
			\[ \|\varphi_Z - \varphi_{t} \|_2 = \left (\frac{1}{2\pi} \int_{-\pi}^\pi |\varphi_Z(\theta) - \varphi_t(\theta)|^2 \, d\theta \right )^{1/2} \leq \frac{3\cdot 10^{13} d^9}{(1-\lambda)^{10}}\frac{1}{t^{3/4}}.\]
		\end{lemma}
		
		\begin{proof}
			We will break the integral into two parts as we did in the proof of Theorem \ref{theo: local CLT Markov}. Set $\theta_0 = (1-\lambda)^2 \sigma^2 /2708$. For any $\theta \in [-\theta_0,\theta_0]$, inequality (3.33) in \cite{Mann1996Berry} gives 
			\[ |\varphi_Z(\theta)- \varphi_{t}(\theta)| \leq c e^{-t\sigma^2\theta^2/8}\left (683t|\theta|^3 + \frac{20}{\sigma^2}|\theta| \right ),
			\]
			where $c=(1-\lambda)^{-2}$. Therefore, 
			\begin{align}\label{eq: lemma for corollary}
				\int_{-\theta_0}^{\theta_0}  |\varphi_Z(\theta)-\varphi_{t}(\theta)|^2 \, d\theta&\leq 4c^2 \int_0^{\theta_0} e^{-t\sigma^2\theta^2/4}\left (683^2t^2\theta^6 + \frac{400}{\sigma^4}\theta^2 \right ) \, d\theta \\ \nonumber
				&= 4c^2 \frac{2}{\sigma \sqrt{t}} \int_0^{\theta_0 \sigma \sqrt{t}/2} e^{-x^2} \left (\frac{683^2 \cdot 2^6 t^2}{\sigma^6 t^3} x^6 + \frac{400\cdot 2^2}{\sigma^4 \cdot \sigma^2 t} x^2\right ) \, dx\\ \nonumber
				&\leq \frac{2^9c^2}{\sigma^7 t^{3/2}} \left (683^2 \int_0^\infty x^6e^{-x^2} \,dx + 25\int_{0}^{\infty} x^2e^{-x^2}\, dx\right )\leq \frac{5\cdot 10^8c^2}{\sigma^7 t^{3/2}}. \nonumber
			\end{align}
			It remains to study the case $\theta_0<|\theta| \leq \pi$. Observe that
			\[\int_{\theta_0\leq |\theta|\leq \pi} | \varphi_Z(\theta)-\varphi_{t}(\theta)|^2\, d\theta \leq 2\int_{\theta_0\leq |\theta|\leq \pi} |\varphi_Z(\theta)|^2\, d\theta + 2\int_{\theta_0\leq |\theta|\leq \pi} |\varphi_{t}(\theta)|^2\, d\theta.\]
			Consider $Z'_t=S_t+Y_t$, where $S_t$ and $Y_t$ are defined as in (\ref{eq: S_t and Y_t definition}) and let $\varphi_{Z'}$ be its characteristic function. Recall that (\ref{eq: bound decomposition}) gives $|\varphi_{Z'}(\theta)|\leq \frac{1}{e \eta b_t}$ for $\theta_0\leq |\theta|\leq \pi$. Moreover, the proof of Theorem \ref{theo: local CLT Markov} shows that 
			\[ |\varphi_Z(\theta) - \varphi_{Z'}(\theta)| \leq \pi \mathbb{E}_{\pi}(|Z_t-Z'_t|) \leq \frac{\pi M}{t} \quad \forall \, \theta \in [-\pi,\pi],\]
			where $M=\frac{10^{10}d^3}{(1-\lambda)^4}$ in view of Lemma \ref{lem: high probability}. Consequently,
			\[ \int_{\theta_0\leq |\theta|\leq \pi} |\varphi_Z(\theta)|^2\, d\theta \leq \frac{4\pi}{e^2 \eta^2 b_t^2} + \frac{2\pi^3 M^2}{t^2}. \]
			On the other hand, a simple computation shows that the function $f(t)=\sqrt{t}e^{-\alpha t}$ attains its maximum at $t=1/(2\alpha)$ and $f(1/(2\alpha))\leq 1/\sqrt{\alpha}$. Taking $\alpha=\sigma^2\theta_0^2$, this implies that $e^{-t\sigma^2 \theta^2}\leq \frac{1}{\sigma \theta_0 \sqrt{t}}$. Therefore,
			\begin{align*} \int_{\theta_0\leq |\theta|\leq \pi} |\varphi_{t}(\theta)|^2\, d\theta &= 2\int_{\theta_0}^\pi e^{-t\sigma^2 \theta^2} \, d\theta = \frac{2}{\sigma\sqrt{t}} \int_{\theta_0 \sigma \sqrt{t}}^{\pi \sigma \sqrt{t}} e^{-x^2}\, dx \leq \frac{2}{\sigma \sqrt{t}} \frac{1}{2\theta_0 \sigma \sqrt{t}} \int_{\theta_0 \sigma \sqrt{t}}^\infty 2x e^{-x^2}\, dx\\
				&= \frac{1}{\theta_0\sigma^2 t} e^{-\theta_0^2 \sigma^2 t} \leq \frac{1}{\theta_0\sigma^2 t} \frac{1}{\theta_0 \sigma \sqrt{t}} = \frac{1}{\theta_0^2 \sigma^3 t^{3/2}}.
			\end{align*}
			Putting everything together gives
			\[ \|\varphi_Z - \varphi_t\|_2^2 = \frac{1}{2\pi} \int_{-\pi}^{\pi} |\varphi_Z(\theta) - \varphi_t(\theta)|^2\, d\theta \leq \frac{10^8c^2}{\sigma^7 t^{3/2}} + \frac{4}{e^2 \eta^2 b_t^2} + \frac{2\pi^2 M^2}{t^2} + \frac{1}{\pi\theta_0^2 \sigma^3 t^{3/2}}. \]
			Taking square root in both sides and using that $\sqrt{a+b} \leq \sqrt{a}+\sqrt{b}$ yields
			\[\|\varphi_Z - \varphi_t\|_2 \leq \frac{10^4 c}{\sigma^{7/2} t^{3/4}} + \frac{2}{e \eta b_t} + \frac{\sqrt{2}\pi M}{t} + \frac{1}{\pi^{1/2}\theta_0 \sigma^{3/2} t^{3/4}}. \]
			
			Recall that $b_t$, $\sigma$, and $\eta$, are bounded in (\ref{eq: b_t}), (\ref{eq: bound sigma}), and (\ref{eq: bound eta}), respectively. The desired result follows from these bounds and the above inequality after straightforward computations.
		\end{proof}
		
		The proof of Corollary \ref{cor:main} does not optimize the constant, but it shows that we can take
		\[ C_2(\lambda,d) = \frac{10^{14} d^9}{(1-\lambda)^{41/4}}.\]

	\begin{proof}[Proof of Corollary \ref{cor:main}]
		Given $c\geq 1$, set $I_1=\{k \in \mathbb{Z} \colon |k-t/2|\leq c\sqrt{t}+1\}$ and $I_2=\mathbb{Z}\setminus I_1$. Recall that $\phi$ stands for the density function of a standard normal distribution and write $\phi_t$ for the density function of a normal distribution with mean $t/2$ and variance $t\sigma^2$, where $\sigma^2=\lim_{t \to \infty} \Var(Z_t)/t$. First, we bound
		\[ \sum_{k\in \mathbb{Z}} \left |\prob{Z_t=k} - \phi_t(k)\right |
		\]
		by breaking the sum using $I_1$ and $I_2$. To bound the sum on $I_2$ we study the tails of the distributions of $Z_t$ and $\mathcal{N}(t/2,t\sigma^2)$. On one hand,
		\begin{align*}
			\sum_{k \in I_2} \phi_t(k) \leq 2\int_{t/2+c\sqrt{t}}^{\infty} \phi_t(x) \, dx
			=2 \int_{c/\sigma}^\infty \phi(y) \, dy
			\leq \frac{2}{\sqrt{2\pi}} \int_{c/\sigma}^{\infty} \frac{y}{c/\sigma} e^{-y^2/2}\, dy = \frac{2\sigma}{c\sqrt{2\pi}} e^{-\frac{c^2}{2\sigma^2}}.
		\end{align*}	
		On the other hand, Lemma \ref{lem:chernoff} gives
		
		\[ \sum_{k \in I_2} \prob{Z_t=k} = \prob{|Z_t-t/2|> c\sqrt{t}+1} \leq 4e^{-\frac{(c\sqrt{t}+1)^2(1-\lambda)}{20t}} \leq 4e^{-\frac{c^2(1-\lambda)}{20}}.\]
		Consequently, we have 
		\begin{equation}\label{eq: cor sum on I2} \sum_{k\in I_2} \left |\prob{Z_t=k} -\phi_t(k) \right |\leq \frac{2\sigma}{c\sqrt{2\pi}} e^{-\frac{c^2}{2\sigma^2}}+ 4e^{-\frac{c^2(1-\lambda)}{20}}.
		\end{equation}
		We can use the Cauchy-Schwarz inequality to bound the sum on $I_1$ as follows.
		\[  \sum_{k\in I_1} \left |\prob{Z_t=k} - \phi_t(k)\right | \leq |I_1|^{1/2} \left ( \sum_{k\in I_1} \left |\prob{Z_t=k} -\phi_t(k)\right |^2  \right )^{1/2}.\]
		We clearly have $|I_1|\leq 2c\sqrt{t}+3\leq 5c\sqrt{t}$. To estimate the sum on the right hand side, we will use Parseval's identity (see \cite[\S 1.4]{helson2010harmonic}). Define $F \colon [-\pi,\pi] \longrightarrow \mathbb{C}$ by 
		\[ F(\theta)=\sum_{k \in \mathbb{Z}} \varphi_t(\theta+2\pi k) \quad \forall \, \theta \in [-\pi,\pi],\]
		where $\varphi_t$ denotes the ch.f. of the normal distribution with mean $t/2$ and variance $t\sigma^2$. Then the Fourier coefficients of $F$, denoted by $a_k(F)$, satisfy 
		\[ a_k(F)=\frac{1}{2\pi} \widehat{\varphi_t}(k)=\phi_t(k) \quad \forall \, k \in \mathbb{Z},\]
		where $\widehat{\varphi_t}$ is the Fourier transform of $\varphi_t$ as defined in \cite[\S 1.2]{helson2010harmonic}. The first equality follows from (2.4.7) in \cite{helson2010harmonic}, and the second one follows from the inversion formula (see Theorem 3.3.14 in \cite{durret2019probability}). Consequently, Parseval's identity gives 
		\[ \sum_{k\in \mathbb{Z}} \left |\prob{Z_t=k} - \phi_t(k)\right |^2 = \|\varphi_Z - F\|_2^2. \]
		Recall that Lemma \ref{lem: L2 norm of varphiZ and varphit} gives a bound for $\|\varphi_Z-\varphi_t\|_2$, so by triangle inequality it suffices to estimate $\|\varphi_t - F\|_2$. For any $\theta \in [-\pi,\pi]$ we have
		\[ |\varphi_t(\theta) - F(\theta)| \leq \sum_{k\neq 0} |\varphi_t(\theta + 2\pi k)| = \sum_{k\neq 0} e^{-\frac{\sigma^2 t(\theta+2\pi k)^2}{2}} \leq 2\sum_{k=1}^\infty e^{-\sigma^2 t k}= 2e^{-\sigma^2 t} \sum_{k=0}^{\infty} (e^{-\sigma^2 t})^k = \frac{2e^{-\sigma^2 t}}{1-e^{-\sigma^2 t}}.\]
		It is easy to check that $e^x\geq 1+x$ for any $x \geq 0$, or equivalently, $xe^{-x}\leq 1-e^{-x}$ for any $x\geq 0$. In view of the above bound, dividing by $x$ and $1-e^{-x}$ both sides and taking $x=\sigma^2 t$ shows that
		\[|\varphi_t(\theta) - F(\theta)| \leq \frac{2e^{-\sigma^2 t}}{1-e^{-\sigma^2 t}} \leq \frac{2}{\sigma^2 t} \leq \frac{200 d^3}{(1-\lambda)^2} \frac{1}{t} \quad \forall \, \theta  \in [-\pi,\pi].\]
		Therefore, we can apply Lemma \ref{lem: L2 norm of varphiZ and varphit} to get 
		\begin{equation}\label{eq: bound norm L2}
			\|\varphi_Z - F\|_2\leq \|\varphi_Z - \varphi_t\|_2 + \|\varphi_t - F\|_2 \leq \frac{4\cdot 10^{13} d^9}{(1-\lambda)^{10}}\frac{1}{t^{3/4}}.
		\end{equation}
		In view of (\ref{eq: cor sum on I2}) and (\ref{eq: bound norm L2}), we conclude that
		\[ \sum_{k\in \mathbb{Z}} \left |\prob{Z_t=k} - \phi_t(k)\right | \leq \frac{2\sigma}{c\sqrt{2\pi}} e^{-\frac{c^2}{2\sigma^2}}+ 4e^{-\frac{c^2(1-\lambda)}{20}} + |5c\sqrt{t}|^{1/2} \frac{4\cdot 10^{13} d^9}{(1-\lambda)^{10}}\frac{1}{t^{3/4}}.
		\]
		Notice that Proposition \ref{prop: variance} implies $\sigma\leq (1-\lambda)^{-1/2}$. Taking $c=\sqrt{10}(1-\lambda)^{-1/2}\sqrt{\log t}$ above gives 
		\begin{align*} \sum_{k\in \mathbb{Z}} \left |\prob{Z_t=k} - \phi_t(k)\right | \leq \frac{1}{t}+ \frac{4}{\sqrt{t}} + \frac{4}{(1-\lambda)^{1/4}} t^{1/4} \log(t)^{1/4} \frac{4\cdot 10^{13} d^9}{(1-\lambda)^{10}}\frac{1}{t^{3/4}}
		\leq \frac{1.7\cdot 10^{14} d^9}{(1-\lambda)^{41/4}} \frac{\log(t)^{1/4}}{\sqrt{t}}.
		\end{align*}
		Finally, we proceed to study the total variation distance 
		\[\|Z_t - N_d(t/2,t\sigma^2)\|_{TV} = \frac{1}{2}\sum_{k\in \mathbb{Z}} \left |\prob{Z_t=k} -f_{N_d(t/2,t\sigma^2)}(k) \right |.\]
		Using the triangle inequality and (\ref{eq:disc normal}) we can upper bound the previous sum by
		\begin{equation}\label{eq: main cor step 2}\sum_{k\in \mathbb{Z}} \left |\prob{Z_t=k} -\phi_t(k)\right | + \left |1-\frac{1}{D(\mu,t\sigma^2)}\right |\sum_{k \in \mathbb{Z}} \phi_t(k).
		\end{equation}
		A bound for the first term in (\ref{eq: main cor step 2}) is given above. For the second term, we can apply Lemma \ref{lem: bound normalizing constant} to get
		\[\left |1-\frac{1}{D(\mu,t\sigma^2)}\right |\sum_{k \in \mathbb{Z}} \phi_t(k) = \left |1-\frac{1}{D(\mu,t\sigma^2)}\right | D(\mu,t\sigma^2) = |D(\mu,t\sigma^2)-1| \leq \frac{1}{\sqrt{2\pi}} \frac{1}{\sigma\sqrt{t}}.\qedhere\]
	\end{proof}
		
		\section{Proof of Theorem \ref{theo: sticky}}\label{sec:sticky}
		
		This section is devoted to prove Theorem \ref{theo: sticky}. Recall that $Z_t$ and $R_t$ denote the Hamming weights of the random walk on a expander graph and the sticky random walk, respectively. It is easy to show that $R_t$ satisfies a local central limit theorem. Since $R_t$ is concentrated around its mean, this implies convergence in total variation distance to a discretized normal distribution. In view of Corollary \ref{cor:main}, we just need to match the means and variances of $Z_t$ and $R_t$ to obtain the result. The mean and variance of a sticky random walk on $\{0,1\}$ are easy to calculate. We do it in the next lemma.
		
		\begin{lemma}
			Let $(Q_i)$ be the sticky random walk on $\{0,1\}$ with parameter $p \in (-1,1)$ starting from stationary distribution $\pi$, and let $R_t=\sum_{i=0}^{t-1} Q_i$. Then $\mathbb{E}_\pi(R_t)=\frac{t}{2}$ and 
			\[ \Var(R_t)=\frac{p^{t+1} -p(t+1)+t}{2(1-p)^2}-\frac{t}{4}.\]
			In particular, $\lim_{t \to \infty} \Var(R_t)/t=\frac{1}{4}\frac{1+p}{1-p}$.
		\end{lemma}
		
		\begin{proof}
		Recall that $Q_0$ is chosen uniformly at random on $\{0,1\}$. Hence, $\mathbb{E}_\pi(R_t)=\sum_{i=0}^{t-1} \mathbb{E}_\pi(Q_i)=\frac{t}{2}$. To calculate the variance,
		consider $P$ the transition matrix of $(Q_i)$, that is,
		\[ P = 
		\begin{pmatrix}
			\frac{1+p}{2} & \frac{1-p}{2} \\
			\frac{1-p}{2} & \frac{1+p}{2} 
		\end{pmatrix}.
		\]
		For any $k \in \mathbb{N}$, we can multiply $P$ by itself $k$ times to obtain that $P^k(1,1)=\frac{1+p^k}{2}$. Consequently, 
		\[ \mathbb{E}_\pi(Q_0 Q_k)=\prob{Q_0=1, Q_k=1}=\frac{1+p^k}{4}.\]
		Using the Markov condition we have
		\begin{align*}
			\mathbb{E}_\pi(R_t^2)&= \sum_{k=0}^{t-1} \mathbb{E}_\pi(Q_k^2) + 2\sum_{k=0}^{t-1} \sum_{j=k+1}^{t-1} \mathbb{E}_\pi(Q_k Q_j)
			= \sum_{k=0}^{t-1} \mathbb{E}_\pi(Q_0^2) + 2\sum_{k=1}^{t-1}(t-k)\mathbb{E}_\pi(Q_0 Q_k)\\
			&= \frac{t}{2} + \frac{1}{2}\sum_{k=1}^{t-1}(t-k)+\frac{1}{2}\sum_{k=1}^{t-1}(t-k)p^k
			=\frac{t^2}{4} - \frac{t}{4} + \frac{1}{2}\sum_{k=0}^{t-1}(t-k)p^k= \frac{p^{t+1} -p(t+1)+t}{2(1-p)^2}-\frac{t}{4} + \frac{t^2}{4}.\qedhere
		\end{align*}
		\end{proof}
	
		The next result shows that the local central limit theorem holds for sticky random walks. We obtain it as an application of Theorem \ref{theo: local CLT Markov}.
		
		\begin{lemma}\label{lemma:localsticky}
			Let $(Q_i)$ be the sticky random walk on $\{0,1\}$ with parameter $p \in (-1,1)$ starting from stationary distribution $\pi$, and let $R_t=\sum_{i=0}^{t-1} Q_i$. For $\sigma^2=\frac{1}{4}\frac{1+p}{1-p}$ we have 
			\[ \left  |\prob{R_t=k} - \frac{1}{\sigma\sqrt{t}}\phi\left (\frac{k-t/2}{\sigma\sqrt{t}}\right )\right | \leq \frac{ 10^{11}}{(1-|p|)^7} \frac{1}{t} \quad \forall \, k \in \mathbb{Z} \quad \forall \, t \in \mathbb{N}.
			\]
		\end{lemma}
	
		\begin{proof}
			We will decompose $R_t$ into a sum of two independent random variables and use Theorem \ref{theo: local CLT Markov} to obtain the result. Let $(Q_i^2)$ be the $2$-steps Markov chain, that is, the markov chain with transition matrix $P^2$. Let $N_t$ be a random variable that counts the number of times that $Q_i^2\neq Q_{i+1}^2$ within $(Q_0,\ldots,Q_{t-1})$. Let $I=\{i \colon Q_i^2\neq Q_{i+1}^2\}$ and denote its elements as $i_1,\ldots, i_{N_t}$. Since $\prob{Q_i^2\neq Q_{i+1}^2}=\frac{1-p^2}{2}$ independently of previous values of the chain, we deduce that $N_t$ follows a binomial $\operatorname{Bin}(\lfloor \frac{t-1}{2} \rfloor, \frac{1-p^2}{2})$. For every $k \in \{1,\ldots,N_t\}$, let $V_k=Q_{2i_k+1}$ be the bit that we skip to go from $Q_{i_k}^2$ to $Q_{i_k+1}^2$. Then $(V_k)$ is a sequence of independent Bernoulli random variables with parameter $\frac{1}{2}$. Take $b_t=\left \lfloor \frac{\mathbb{E}_\pi(N_t)}{2}\right  \rfloor$ and define the random variables
			\[ S'_t=\sum_{k=1}^{\min\{b_t,N_t\}} V_k, \quad Y_t=R_t - S'_t, \quad S_t=S'_t + \sum_{k=1}^{(b_t-N_t)^+} V'_k,\]
			where $(V'_k)$ are Bernoulli random variables with parameter $\frac{1}{2}$ independent of everything else. We claim that $(Y_t,S_t,R_t)$ satisfies the hypotheses of Theorem \ref{theo: local CLT Markov}. First, notice that $S_t$ and $Y_t$ are independent. In fact, once we know that the $Q_{2i}\neq Q_{2i+2}$, the value of $Q_{2i+1}$ does not affect the rest of the chain. Moreover, we can repeat the proof of Lemma \ref{lem: high probability} to obtain $\mathbb{E}_\pi(|R_t-S_t-Y_t|) \leq M/t$, where
			\[M=\frac{779}{(1-|p|)^2}.\]  
			Indeed, since the function $f(t)=t^2e^{-\alpha t}$ attains its maximum at $t=2\alpha^{-1}$, Lemma \ref{lem: binomial} implies
			\[ \prob{N_t \leq b_t}\leq e^{-\mathbb{E}_\pi(N_t) \varphi(1/2)}\leq e^{-\frac{(1-p^2)\varphi(1/2)}{4}(t-3)} \leq e^{3/4}\frac{4^3e^{-2}}{(1-p^2)^2\varphi(1/2)^2}  \frac{1}{t^2} \leq \frac{779}{(1-|p|)^2}\frac{1}{t^2}.\]
			Hence, the claim follows from the fact that $|R_t-S_t-Y_t|\leq t$ and that $R_t=S_t+Y_t$ if $b_t \leq N_t$. Finally, the eigenvalues of the sticky random walk with parameter $p$ are $1$ and $p$, so in this case we have $1-\lambda = 1-|p|$. Let $\theta_0=(1-|p|)^2\sigma^2/2708$. Since $\sigma^2 \geq (1-|p|)/8$, Lemma \ref{lem: nonlattice} shows that the random variables $V_i$ and $V'_i$ are $\eta$-nonlattice with
			\[\eta \geq  \frac{1}{4}(1-\cos(\theta_0))\geq \frac{1}{4} \frac{\theta_0^2}{5} \geq \frac{(1-|p|)^6}{20\cdot 2708^2 \cdot 64} \geq \frac{(1-|p|)^6}{10^{10}}. \]
			Therefore, Theorem \ref{theo: local CLT Markov} gives
			\begin{equation}\label{eq: sticky markov} \left  |\prob{R_t=k} - \frac{1}{\sigma\sqrt{t}}\phi\left (\frac{k-t/2}{\sigma\sqrt{t}}\right )\right | \leq \left (\pi M +\frac{1}{\theta_0 \sigma^2} + \frac{C_4}{\sigma^4(1-|p|)^2} \right ) \frac{1}{t} + \frac{1}{e\eta} \frac{1}{b_t}.
			\end{equation}
			If $t\leq 22(1-p^2)^{-1}$, then there is nothing to prove. Otherwise, we have
			\[ b_t=\left  \lfloor \frac{\mathbb{E}_\pi(N_t)}{2} \right  \rfloor \geq  \frac{\left \lfloor \frac{t-1}{2} \right \rfloor \frac{1-p^2}{2}}{2} -1 \geq \frac{(t-3)(1-p^2)}{8}-1 \geq \frac{t(1-p^2) - 11}{8} \geq \frac{(1-p^2)}{16}t\geq \frac{(1-|p|)}{16}t. \]
			Substituting all previous bounds in (\ref{eq: sticky markov}) gives
			\[ \left  |\prob{R_t=k} - \frac{1}{\sigma\sqrt{t}}\phi\left (\frac{k-t/2}{\sigma\sqrt{t}}\right )\right | \leq \left ( \frac{779\pi}{(1-|p|)^2} + \frac{10^6}{(1-|p|)^4} + \frac{8 \cdot 10^{10}}{(1-|p|)^7} \right ) \frac{1}{t}\leq \frac{10^{11}}{(1-|p|)^7} \frac{1}{t}.\qedhere
			\]
		\end{proof}
	
		We can repeat word by word the proof of Corollary \ref{cor:main} using $R_t$ instead of $Z_t$ and Lemma \ref{lemma:localsticky} instead of Theorem \ref{theo:main} to produce the following result.
		
		\begin{lemma}\label{lem: sticky tvd}
			Let $(Q_i)$ be the sticky random walk on $\{0,1\}$ with parameter $p \in (-1,1)$ starting from stationary distribution, and let $R_t=\sum_{i=0}^{t-1} Q_i$. For $\sigma^2=\frac{1}{4}\frac{1+p}{1-p}$ we have
			\[ \left \|R_t - N_d (t/2,t\sigma^2) \right \|_{TV} \leq \frac{ 10^{12}}{(1-|p|)^8} \frac{\sqrt{\log t}}{\sqrt{t}}.\]
		\end{lemma}
		
		\begin{proof}
			As in the proof of Corollary \ref{cor:main}, take $c\geq 1$ and let $I_1=\{u \in \mathbb{Z} \colon |u-t/2|\leq c\sqrt{t}+1\}$ and $I_2=\mathbb{Z}\setminus I_1$. First, we bound
			\[ \sum_{k\in \mathbb{Z}} \left |\prob{R_t=k} - t^{-1/2}\sigma^{-1}\phi\left (\frac{k-t/2}{t^{1/2}\sigma}\right )\right |
			\]
			by breaking the sum using $I_1$ and $I_2$. Write $C(p)=10^{11} (1-|p|)^{-7}$. Lemma \ref{lemma:localsticky} gives
			\[ \sum_{k \in I_1} \left |\prob{R_t=k} - t^{-1/2}\sigma^{-1}\phi\left (\frac{k-t/2}{t^{1/2}\sigma}\right )\right | \leq (2c\sqrt{t} +3) \frac{C(p)}{t} \leq 5C(p) \frac{c}{\sqrt{t}}.\]
			To bound the sum on $I_2$ we study the tails of the distributions of $R_t$ and $\mathcal{N}(t/2,t\sigma^2)$. Since we can also apply Lemma \ref{lem:chernoff} to study the tail of $R_t$, from the proof of Corollary \ref{cor:main} it follows that
			\[ \sum_{k\in I_2} \left |\prob{R_t=k} - t^{-1/2}\sigma^{-1}\phi\left (\frac{k-t/2}{t^{1/2}\sigma}\right )\right | \leq \frac{2\sigma}{c\sqrt{2\pi}} e^{-\frac{c^2}{2\sigma^2}} + 4e^{-\frac{c^2(1-|p|)}{20}}.\]
			Write $\alpha=\max\{\sigma, \sqrt{10}(1-|p|)^{-1/2}\}$ and take $c= \alpha\sqrt{\log t}$. From the previous bounds we obtain
			\[\sum_{k\in \mathbb{Z}} \left |\prob{R_t=k} - t^{-1/2}\sigma^{-1}\phi\left ( \frac{k-t/2}{t^{1/2}\sigma}\right )\right | \leq \left (5C(p)\alpha \sqrt{\log t}+ \frac{\sqrt{2}}{\sqrt{\pi}} + 4\right ) \frac{1}{\sqrt{t}}.
			\]
			Recall that the total variation distance between $R_r$ and $N_d(t/2,t\sigma^2)$ is given by 
			\[\|R_t - N_d(t/2,t\sigma^2)\|_{TV} = \frac{1}{2}\sum_{k\in \mathbb{Z}} \left |\prob{R_t=k} -f_{N_d(t/2,t\sigma^2)}(k) \right |.\]
			We can use the triangle inequality and repeat the argument in the proof of Corollary \ref{cor:main} to bound (\ref{eq: main cor step 2}) to obtain that
			\[ \sum_{k\in \mathbb{Z}} \left |\prob{R_t=k} -f_{N_d(t/2,t\sigma^2)}(k) \right | \leq \left (5C(p)\alpha \sqrt{\log t}+ \frac{\sqrt{2}}{\sqrt{\pi}} + 4\right ) \frac{1}{\sqrt{t}} + \frac{1}{\sqrt{2\pi}} \frac{1}{\sigma\sqrt{t}}.\]
			We clearly have $\sigma\leq (1-|p|)^{-1}$, whence $\alpha \leq \sqrt{10} (1-|p|)^{-1}$. We conclude that
			\[ \|R_t - N_d(t/2,t\sigma^2)\|_{TV} \leq \frac{1}{2} \left (5C_1(\lambda,d)\alpha \sqrt{\log t}+ \frac{\sqrt{2}}{\sqrt{\pi}} + 4 + \frac{1}{\sigma\sqrt{2\pi}}\right )\frac{1}{\sqrt{t}} \leq \frac{3\sqrt{10}C(p)}{1-|p|} \frac{\sqrt{\log t}}{\sqrt{t}}.\qedhere\]
		\end{proof}

		The proof of Theorem \ref{theo: sticky} is now immediate. Although we do not optimize the constant, we show that Theorem \ref{theo: sticky} holds with
		\[ C_3(\lambda,d) = \frac{2\cdot 10^{23}d^{24}}{(1-\lambda)^{16}}.\]
		
		\begin{proof}[Proof of Theorem \ref{theo: sticky}]
			We want to write the bound of Lemma \ref{lem: sticky tvd} in terms of $\lambda$ and $d$. Write $\sigma^2= \lim_{t \to \infty} \Var(Z_t)/t$ and recall that $p$ is chosen so that $\sigma^2=\frac{1}{4} \frac{1+p}{1-p}$, that is, 
			\[ p = \frac{4\sigma^2 - 1}{1+4\sigma^2}.\]
			If $\sigma^2 \in (0,1/4)$, then we have $|p|\leq 1-4\sigma^2$, and if $\sigma^2 \in (1/4,\infty)$, then $|p| \leq 1 - \frac{1}{4\sigma^2}$. Therefore,
			\[ 1-|p| \geq \min\left \{4\sigma^2, \frac{1}{4\sigma^2}\right \}.\]
			Recall that the bound (\ref{eq: bound sigma}) gives $4\sigma^2 \geq \frac{(1-\lambda)^2}{25d^3}$. Moreover, Propositon \ref{prop: variance} implies $4\sigma^2\leq 2(1-\lambda)^{-1}$, so we have
			\[ 1-|p| \geq \frac{(1-\lambda)^2}{25d^3}.
			\]
			Therefore, Lemma \ref{lem: sticky tvd} gives
			\[ \left \|R_t - N_d (t/2,t\sigma^2) \right \|_{TV} \leq \frac{10^{12}\cdot 25^8d^{24}}{(1-\lambda)^{16}} \frac{\sqrt{\log t}}{\sqrt{t}} \leq \frac{1.6 \cdot 10^{23}d^{24}}{(1-\lambda)^{16}}\frac{\sqrt{\log t}}{\sqrt{t}} .\]
			The result now follows from Corollary \ref{cor:main} and the triangle inequality.
		\end{proof}

		\section{Generalization to all labellings}\label{sec:extension}
		
		In this section, we extend Theorem \ref{theo:main} and Corollary \ref{cor:main} to allow unbalanced labellings. First, we need a generalization of Lemma \ref{lem:A_j or B_j}. Recall that $A=\{x \in V \colon \Val(x)=0\}$ and $B=\stcomp{A}$.
		\begin{lemma}\label{lem:A_j or B_j extended}
			Let $G$ be a $d$-regular $\lambda$-expander graph with $n$ vertices. Fix a labelling $\operatorname{val}\colon V \longrightarrow \{0,1\}$ on $G$ with $\mathbb{E}_\pi(\operatorname{val})=\alpha \in [0,1]$. Write $\delta=\frac{1}{4}(1-\lambda)^2\alpha(1-\alpha)$. Then there is $k^* \in \{1,\ldots,d-1\}$ such that either 
			$$|A_{k^*}|\geq \frac{\delta(1-\alpha)n}{d-1}\quad \mbox{or} \quad |B_{k^*}|\geq \frac{\delta\alpha n}{d-1}.$$
		\end{lemma}
		
		\begin{proof}
			We will follow the argument in the proof of Lemma \ref{lem:A_j or B_j}. Suppose the statement is false. Then
			\begin{equation}\label{eq:|A_0|+|A_d| extended} 
				|A_0|+|A_d|>(1-\delta)(1-\alpha)n \quad \mbox{and} \quad |B_0|+|B_d|> (1-\delta)\alpha n.
			\end{equation}
			Take $F_1=A$ and $F_2=B$. Corollary \ref{cor:expander mixing lemma} gives $|E(A,B)|\geq \frac{1}{2}(1-\lambda) d \alpha(1-\alpha) n$. Notice also that 
			$$|E(A,B)|\leq d |A\setminus A_d|\leq d(|A_0|+\delta|A|).$$
			Therefore, we obtain
			\begin{equation}\label{eq:Ad extended} 
				|A_0|\geq \frac{1}{2}(1-\lambda)\alpha(1-\alpha) n - \delta(1-\alpha)n=\frac{1}{2}(1-\lambda)\alpha(1-\alpha) n \left (1-\frac{1}{2}(1-\lambda)(1-\alpha)\right ) \geq \frac{1}{4}(1-\lambda)\alpha(1-\alpha) n.
			\end{equation}
			A completely analogous argument replacing $A_d$ with $B_0$ and $A_0$ with $B_d$ gives
			\begin{equation}\label{eq:B0 extended} 
				|B_d|\geq \frac{1}{2}(1-\lambda)\alpha(1-\alpha) n - \delta \alpha n \geq \frac{1}{4}(1-\lambda) \alpha(1-\alpha) n.
			\end{equation}
			Next, consider $F_1=F_2=A$. Then the expander mixing lemma \ref{lem:expander mixing lemma} gives 
			\[ |E(A,A)|\geq (1-\alpha)^2dn - \alpha(1-\alpha)\lambda dn = ((1-\alpha)-\alpha\lambda)(1-\alpha)dn.\]
			Moreover, $|E(A,A)|\leq d|A\setminus A_0| \leq d(|A_d|+\delta|A|)$, from where we deduce that
			\[ |A_d|\geq ((1-\alpha)-\alpha\lambda -\delta)(1-\alpha)n.\]
			We can take $F_1=F_2=B$ and repeat the previous argument, replacing $A_d$ with $B_0$ and $A_0$ with $B_d$, to obtain that
			\[ |B_0|\geq (\alpha -(1-\alpha)\lambda -\delta) \alpha n.\]
			Adding these bounds gives
			\begin{align*}
				|A_d|+|B_0|&\geq \left ((1-\alpha)^2 -2\alpha(1-\alpha)\lambda + \alpha^2 \right )n - \delta n = ((1-\alpha) - \alpha)^2 n + 2\alpha(1-\alpha)(1-\lambda)n -\delta n\\
				&\geq 2\alpha(1-\alpha)(1-\lambda)n - \delta n \geq \frac{1}{2} \alpha(1-\alpha)(1-\lambda)n.
			\end{align*}
			Finally, take $F_1=A_0\cup B_d$ and $F_2=\stcomp{F_1}$. In view of the above bound, (\ref{eq:Ad extended}), and (\ref{eq:B0 extended}), we get $\min\{|F_1|,|F_2|\}\geq \frac{1}{2}(1-\lambda)\alpha(1-\alpha) n$. Hence, Corollary \ref{cor:expander mixing lemma} gives
			\begin{align*} |E(F_1,F_2)|&\geq \frac{1}{4}d (1-\lambda)^2\alpha(1-\alpha) n=d\delta n.
			\end{align*}
			Therefore, we must have either $|E(A_0,F_2)|\geq d \delta\alpha n$ or $|E(B_d,F_2)|\geq d \delta(1-\alpha) n$. In the first case, for any $e=\{x,y\} \in E(A_0,F_2)$ with $x \in A_0$, we must have $\Val(y)=1$. Thus, $y \in B\setminus B_d$. Moreover, $y$ cannot belong to $B_0$ since it is adjacent to $x$ and $\Val(x)=0$. Therefore, $y \in B\setminus(B_0\cup B_d)$. Consequently,
			\[ |B\setminus(B_0\cup B_d)|\geq \frac{|E(A_0,F_2)|}{d}\geq  \delta \alpha n,\]
			which contradicts (\ref{eq:|A_0|+|A_d| extended}). Analogously, $|E(B_d,F_2)|\geq d \delta(1-\alpha) n$ also leads to a contradiction.\qedhere
		\end{proof}
		
		To extend our main result for unbalanced labellings we just need to repeat its proof using Lemma \ref{lem:A_j or B_j extended} instead of Lemma \ref{lem:A_j or B_j}.
	
			\begin{theorem}\label{theo:main extension}
			Let $G$ be a $d$-regular $\lambda$-expander graph with $\lambda <1$. Let $(X_i)$ be the simple random walk on $G$ with uniform initial distribution $\pi$, fix a labelling $\operatorname{val}\colon V \longrightarrow \{0,1\}$ with $\alpha=\mathbb{E}_\pi(\Val) \in (0,1)$, and let $Z_t= \sum_{i=0}^{t-1} \Val(X_i)$ and $\sigma^2=\lim_{t\to \infty} \Var(Z_t)/t$. There is a constant $C_5(\lambda,d,\alpha)$ depending on $\lambda$, $d$, and $\alpha$ such that
			\[ \left |\prob{Z_t =k} - t^{-1/2}\sigma^{-1}\phi\left (\frac{k-\alpha t}{t^{1/2}\sigma}\right )\right | \leq C_5(\lambda,d,\alpha) \frac{1}{t} \quad \forall \, k \in \mathbb{Z} \quad \forall\, t \in \mathbb{N}. \]
		\end{theorem}
		
		Although the proof of Theorem \ref{theo:main extension} does not optimize the constant, it shows that we can take
		\[ C_5(\lambda,d,\alpha)= \frac{4\cdot 10^{12} d^9}{\alpha^3 (1-\alpha)^6 (1-\lambda)^{10}}.\]
	
		\begin{proof}
			Let $\delta=\frac{1}{4}(1-\lambda)^2 \alpha (1-\alpha)$. Lemma \ref{lem:A_j or B_j extended} guarantees that there is $k^* \in \{1,\ldots,d-1\}$ such that either 
			$$|A_{k^*}|\geq \frac{\delta(1-\alpha)n}{d-1}\quad \mbox{or} \quad |B_{k^*}|\geq \frac{\delta\alpha n}{d-1}.$$
			By symmetry, we may assume that $|A_{k^*}|\geq \frac{\delta(1-\alpha)n}{d-1}$. Fix $t \in \mathbb{N}$ and consider the random variables  $\widetilde{N}_t$, $V_i$, $\widetilde{V}_i$, $S'_t$, $S_t$, and $Y_t$ appearing in (\ref{eq: S_t and Y_t definition}). As the proof of Theorem \ref{theo:main} showed, the hypotheses of Theorem \ref{theo: local CLT Markov} are satisfied, so for any $k \in \mathbb{Z}$ we have 
			\begin{equation}\label{eq: theorem 4.1 applied extension} \left  |\prob{Z_t=k} - \frac{1}{\sigma\sqrt{t}}\phi\left (\frac{k-\alpha t}{\sigma\sqrt{t}}\right )\right | \leq \left (\pi M +\frac{1}{\theta_0 \sigma^2} + \frac{C_4}{\sigma^4(1-\lambda)^2} \right )\frac{1}{t} + \frac{1}{e\eta} \frac{1}{b_t}.
			\end{equation}
			The value of $M$ is the same as in the proof of Theorem \ref{theo:main}. In contrast, $b_t$ now depends on $\alpha$, so we need to find a new bound for it. Recall that $b_t=\lfloor \mathbb{E}_\pi(\widetilde{N}_t)/4 \rfloor$, where
			\[
			\frac{\mathbb{E}_\pi(\widetilde{N}_t)}{4}= \frac{\mathbb{E}_\pi(N_t)}{4d} = \frac{ \lfloor t/2 \rfloor \pi(A_{k^*})}{4d} \geq \frac{(t-2)\delta(1-\alpha)}{8d(d-1)}\geq \frac{t\delta(1-\alpha)-2}{8d(d-1)}.
			\]
			If $t \leq 8\delta^{-1}(1-\alpha)^{-1}d^3$, then there is nothing to prove. Otherwise, a simple computation shows that
			\begin{equation}\label{eq: bt with alpha}
				b_t=\left \lfloor \frac{\mathbb{E}_\pi(\widetilde{N}_t)}{4}\right  \rfloor \geq \frac{\mathbb{E}_\pi(\widetilde{N}_t)}{4} -1 \geq \frac{t\delta(1-\alpha)-4
					8d^2}{4
					8d(d-1)} \geq \frac{\delta(1-\alpha)}{8d^2}t=\frac{\alpha(1-\alpha)^2(1-\lambda)^2}{32d^2}t.
			\end{equation}
			Thus, the argument used to get the bound (\ref{eq: bound sigma}) for $\sigma$ shows that
			\begin{equation}\label{eq: sigma with alpha} \sigma^2 \geq \frac{b_t}{t} \frac{d-1}{d^2} \geq \frac{\alpha(1-\alpha)^2(1-\lambda)^2}{64d^3}.
			\end{equation}
			Finally, to obtain a bound for $\eta$ take 
			\[	\theta_0= \frac{(1-\lambda)^2\sigma^2}{2708}.\]
			Then following the proof of Theorem \ref{theo:main} we get
			\begin{equation}\label{eq: eta with alpha}
				\eta \geq \frac{d-1}{d^2}(1-\cos(\theta_0)) \geq \frac{1}{2d} \frac{\theta_0^2}{5}\geq \frac{1}{10\cdot 64^2 \cdot 2708^2}\frac{\alpha^2(1-\alpha)^4(1-\lambda)^8}{d^7} \geq \frac{\alpha^2(1-\alpha)^4(1-\lambda)^8}{3.1\cdot 10^{11}d^7}.
			\end{equation}
			After substituting the previous bounds in (\ref{eq: theorem 4.1 applied extension}), a simple computation gives
			\[\left  |\prob{Z_t=k} - \frac{1}{\sigma\sqrt{t}}\phi\left (\frac{k-\alpha t}{\sigma\sqrt{t}}\right )\right | \leq \frac{4\cdot 10^{12} d^9}{\alpha^3 (1-\alpha)^6 (1-\lambda)^{10}} \frac{1}{t}.\qedhere\]
			
		\end{proof}
	
		Similarly, we can repeat word by word the proof of Corollary \ref{cor:main} to extend it for unbalanced labellings.
		
		\begin{corollary}\label{cor:main extension}
			Let $G$ be a $d$-regular $\lambda$-expander graph with $\lambda <1$. Let $(X_i)$ be the simple random walk on $G$ with uniform initial distribution $\pi$, fix a labelling $\operatorname{val}\colon V \longrightarrow \{0,1\}$ with $\alpha=\mathbb{E}_\pi(\Val) \in (0,1)$, and let $Z_t= \sum_{i=0}^{t-1} \Val(X_i)$ and $\sigma^2=\lim_{t\to \infty} \Var(Z_t)/t$. There is a constant $C_6(\lambda,d,\alpha)$ depending on $\lambda$, $d$, and $\alpha$ such that
			\[ \left \|Z_t - N_d(t/2,t\sigma^2) \right \|_{TV}\leq C_6(\lambda,d,\alpha) \frac{\log (t)^{1/4}}{\sqrt{t}} \quad \forall \, t \geq 2.\]
		\end{corollary}
		The proof of Corollary \ref{cor:main extension} does not optimize the constant, but it shows that we can take
		\[ C_6(\lambda,d,\alpha)=\frac{2\cdot 10^{13} d^9}{\alpha^3 (1-\alpha)^6(1-\lambda)^{41/4}}.\]
		We just need to verify that every result used for the proof of Corollary \ref{cor:main} is also valid for unbalanced labellings. First, Lemma \ref{lem: bound normalizing constant} was proved for arbitrary normal distributions, so it can be used in the unbalanced setting. Next, we can repeat word by word the proof of Lemma \ref{lem: L2 norm of varphiZ and varphit} to extend it for unbalanced labellings. For convenience, we use $\varphi_t$ for the characteristic function of a normal distribution with mean $t\alpha$ and variance $t\sigma^2$, where $\sigma^2=\lim_{t \to \infty} \Var(Z_t)/t$.
		
		\begin{lemma}\label{lem: L2 norm of varphiZ and varphit with alpha}
			The characteristic function of $Z_t$ satisfies
			\[ \|\varphi_Z - \varphi_{t} \|_2 = \left (\frac{1}{2\pi} \int_{-\pi}^\pi |\varphi_Z(\theta) - \varphi_t(\theta)|^2 \, d\theta \right )^{1/2} \leq \frac{8\cdot 10^{12} d^9}{\alpha^3 (1-\alpha)^6(1-\lambda)^{10}}\frac{1}{t^{3/4}}.\]
		\end{lemma}
		
		\begin{proof}
		 Set $\theta_0=(1-\lambda)^2\sigma^2/2708$. For any $\theta \in [-\theta_0,\theta_0]$, inequality (3.33) in \cite{Mann1996Berry} gives 
			\[ |\varphi_Z(\theta)- \varphi_{t}(\theta)| \leq c e^{-t\sigma^2\theta^2/8}\left (683t|\theta|^3 + \frac{20}{\sigma^2}|\theta| \right ),
			\]
			where $c=(1-\lambda)^{-2}$. By repeating the computations in (\ref{eq: lemma for corollary}) we obtain
			\begin{align}
				\int_{-\theta_0}^{\theta_0}  |\varphi_Z(\theta)-\varphi_{t}(\theta)|^2 \, d\theta \leq \frac{5\cdot 10^8c^2}{\sigma^7 t^{3/2}}.
			\end{align}
			It remains to study the case $\theta_0<|\theta| \leq \pi$. As we did for Lemma \ref{lem: L2 norm of varphiZ and varphit}, consider $Z'_t=S_t+Y_t$, where $S_t$ and $Y_t$ are defined as in (\ref{eq: S_t and Y_t definition}) and let $\varphi_{Z'}$ be its characteristic function. Recall that (\ref{eq: bound decomposition}) gives $|\varphi_{Z'}(\theta)|\leq \frac{1}{e \eta b_t}$ for $\theta_0\leq |\theta|\leq \pi$. Moreover, the proof of Theorem \ref{theo: local CLT Markov} shows that 
			\[ |\varphi_Z(\theta) - \varphi_{Z'}(\theta)| \leq \pi \mathbb{E}_{\pi}(|Z_t-Z'_t|) \leq \frac{\pi M}{t} \quad \forall \, \theta \in [-\pi,\pi],\]
			where $M=\frac{10^{10}d^3}{(1-\lambda)^4}$ in view of Lemma \ref{lem: high probability}. Consequently,
			\[ \int_{\theta_0\leq |\theta|\leq \pi} |\varphi_Z(\theta)|^2\, d\theta \leq \frac{4\pi}{e^2 \eta^2 b_t^2} + \frac{2\pi^3 M^2}{t^2}. \]
			On the other hand, we showed in the proof of Lemma \ref{lem: L2 norm of varphiZ and varphit} that
			\begin{align*} \int_{\theta_0\leq |\theta|\leq \pi} |\varphi_{t}(\theta)|^2\, d\theta \leq \frac{1}{\theta_0^2 \sigma^3 t^{3/2}}.
			\end{align*}
			Putting everything together gives
			\[ \|\varphi_Z - \varphi_t\|_2^2 = \frac{1}{2\pi} \int_{-\pi}^{\pi} |\varphi_Z(\theta) - \varphi_t(\theta)|^2\, d\theta \leq \frac{10^8c^2}{\sigma^7 t^{3/2}} + \frac{4}{e^2 \eta^2 b_t^2} + \frac{2\pi^2 M^2}{t^2} + \frac{1}{\pi\theta_0^2 \sigma^3 t^{3/2}}. \]
			Taking square root in both sides and using that $\sqrt{a+b} \leq \sqrt{a}+\sqrt{b}$ yields
			\[\|\varphi_Z - \varphi_t\|_2 \leq \frac{10^4 c}{\sigma^{7/2} t^{3/4}} + \frac{2}{e \eta b_t} + \frac{\sqrt{2}\pi M}{t} + \frac{1}{\pi^{1/2}\theta_0 \sigma^{3/2} t^{3/4}}. \]
			Recall that $b_t$, $\sigma$, and $\eta$, are bounded in (\ref{eq: bt with alpha}), (\ref{eq: sigma with alpha}), and (\ref{eq: eta with alpha}), respectively. The claim follows from these bounds and the above inequality after straightforward computations.
		\end{proof}
		
		\begin{proof}[Proof of Corollary \ref{cor:main extension}]
			The argument is analogous to the one in the proof of Corollary \ref{cor:main}. Given $c\geq 1$, set $I_1=\{k \in \mathbb{Z} \colon |k-t\alpha|\leq c\sqrt{t}+1\}$ and $I_2=\mathbb{Z}\setminus I_1$. First, we bound
			\[ \sum_{k\in \mathbb{Z}} \left |\prob{Z_t=k} - \phi_t(k)\right |
			\]
			by breaking the sum using $I_1$ and $I_2$. To bound the sum on $I_2$ we study the tails of the distributions of $Z_t$ and $\mathcal{N}(t\alpha ,t\sigma^2)$. Let $\phi_t$ denote the density function of $\mathcal{N}(t\alpha ,t\sigma^2)$. On one hand, 
			\begin{align*}
				\sum_{k \in I_2} \phi_t(k) \leq 2\int_{t\alpha +c\sqrt{t}}^{\infty} \phi_t(x) \, dx
				=2 \int_{c/\sigma}^\infty \phi(y) \, dy
				\leq \frac{2}{\sqrt{2\pi}} \int_{c/\sigma}^{\infty} \frac{y}{c/\sigma} e^{-y^2/2}\, dy = \frac{2\sigma}{c\sqrt{2\pi}} e^{-\frac{c^2}{2\sigma^2}}.
			\end{align*}	
			On the other hand, Lemma \ref{lem:chernoff} gives
			\[ \sum_{k \in I_2} \prob{Z_t=k} = \prob{|Z_t-t\alpha|> c\sqrt{t}+1} \leq 4e^{-\frac{(c\sqrt{t}+1)^2(1-\lambda)}{20t}} \leq 4e^{-\frac{c^2(1-\lambda)}{20}}.\]
			Consequently, we have 
			\begin{equation}\label{eq: cor sum on I2 with alpha} \sum_{k\in I_2} \left |\prob{Z_t=k} -\phi_t(k) \right |\leq \frac{2\sigma}{c\sqrt{2\pi}} e^{-\frac{c^2}{2\sigma^2}}+ 4e^{-\frac{c^2(1-\lambda)}{20}}.
			\end{equation}
			For the sum on $I_1$, the Cauchy-Schwarz inequality yields
			\[  \sum_{k\in I_1} \left |\prob{Z_t=k} - \phi_t(k)\right | \leq |I_1|^{1/2} \left ( \sum_{k\in I_1} \left |\prob{Z_t=k} -\phi_t(k)\right |^2  \right )^{1/2}.\]
			We have $|I_1|\leq 2c\sqrt{t}+3\leq 5c\sqrt{t}$. As in the proof of Corollary \ref{cor:main}, to estimate the sum on the right hand side we will use Parseval's identity (see \cite[\S 1.4]{helson2010harmonic}). Define $F \colon [-\pi,\pi] \longrightarrow \mathbb{C}$ by 
			\[ F(\theta)=\sum_{k \in \mathbb{Z}} \varphi_t(\theta+2\pi k) \quad \forall \, \theta \in [-\pi,\pi],\]
			where $\varphi_t$ denotes the ch.f. of the normal distribution with mean $t\alpha$ and variance $t\sigma^2$. Then the Fourier coefficients of $F$, denoted by $a_k(F)$, satisfy 
			\[ a_k(F)=\frac{1}{2\pi} \widehat{\varphi_t}(k)=\phi_t(k) \quad \forall \, k \in \mathbb{Z},\]
			where $\widehat{\varphi_t}$ is the Fourier transform of $\varphi_t$ as defined in \cite[\S 1.2]{helson2010harmonic}. Consequently, Parseval's identity gives 
			\[ \sum_{k\in \mathbb{Z}} \left |\prob{Z_t=k} - \phi_t(k)\right |^2 = \|\varphi_Z - F\|_2^2. \]
			By repeating the computations in the proof of Corollary \ref{cor:main} we obtain	
			\[|\varphi_t(\theta) - F(\theta)| \leq \frac{200 d^3}{(1-\lambda)^2} \frac{1}{t} \quad \forall \, \theta  \in [-\pi,\pi].\]
			Therefore, from Lemma \ref{lem: L2 norm of varphiZ and varphit with alpha} and the triangle inequality we deduce that
			\begin{equation}\label{eq: bound norm L2 with alpha}
				\|\varphi_Z - F\|_2\leq \|\varphi_Z - \varphi_t\|_2 + \|\varphi_t - F\|_2 \leq \frac{9\cdot 10^{12} d^9}{\alpha^3 (1-\alpha)^6(1-\lambda)^{10}}\frac{1}{t^{3/4}}.
			\end{equation}
			In view of (\ref{eq: cor sum on I2 with alpha}) and (\ref{eq: bound norm L2 with alpha}), we conclude that
			\[ \sum_{k\in \mathbb{Z}} \left |\prob{Z_t=k} - \phi_t(k)\right | \leq \frac{2\sigma}{c\sqrt{2\pi}} e^{-\frac{c^2}{2\sigma^2}}+ 4e^{-\frac{c^2(1-\lambda)}{20}} + |5c\sqrt{t}|^{1/2} \frac{9\cdot 10^{12} d^9}{\alpha^3 (1-\alpha)^6(1-\lambda)^{10}}\frac{1}{t^{3/4}}.
			\]
			Notice that Proposition \ref{prop: variance} implies $\sigma\leq (1-\lambda)^{-1/2}$. Taking $c=\sqrt{10}(1-\lambda)^{-1/2}\sqrt{\log t}$ above gives 
			\begin{align*} \sum_{k\in \mathbb{Z}} \left |\prob{Z_t=k} - \phi_t(k)\right | \leq  \frac{3.7\cdot 10^{13} d^9}{\alpha^3 (1-\alpha)^6(1-\lambda)^{41/4}} \frac{\log(t)^{1/4}}{\sqrt{t}}.
			\end{align*}
			Finally, we proceed to study the total variation distance 
			\[\|Z_t - N_d(t/2,t\sigma^2)\|_{TV} = \frac{1}{2}\sum_{k\in \mathbb{Z}} \left |\prob{Z_t=k} -f_{N_d(t/2,t\sigma^2)}(k) \right |.\]
			Using the triangle inequality and (\ref{eq:disc normal}) we can upper bound the previous sum by
			\begin{equation}\label{eq: main cor step 2}\sum_{k\in \mathbb{Z}} \left |\prob{Z_t=k} -\phi_t(k)\right | + \left |1-\frac{1}{D(\mu,t\sigma^2)}\right |\sum_{k \in \mathbb{Z}} \phi_t(k).
			\end{equation}
			The first term in (\ref{eq: main cor step 2}) was bounded above. For the second term, we can apply Lemma \ref{lem: bound normalizing constant} to get
			\[\left |1-\frac{1}{D(\mu,t\sigma^2)}\right |\sum_{k \in \mathbb{Z}} \phi_t(k) = \left |1-\frac{1}{D(\mu,t\sigma^2)}\right | D(\mu,t\sigma^2) = |D(\mu,t\sigma^2)-1| \leq \frac{1}{\sqrt{2\pi}} \frac{1}{\sigma\sqrt{t}}.\qedhere\]
		\end{proof}

		\section{Proofs of Example \ref{ex: K4} and Example \ref{ex: uniform}}\label{sec:examples}
		Let $(X_i)$ denote the SRW on a graph $G=(V,E)$ starting from stationary distribution $\pi$, and write $(Y_i)=(\Val(X_i))$, where $\Val\colon V \longrightarrow \{0,1\}$ is a labelling on $G$. First, we prove Example \ref{ex: K4}.
			
		\begin{proof}[Proof of Example \ref{ex: K4}]
				We want to use the formula (\ref{eq:var of sum}) for the variance of a sum of random variables. We claim that
				\[ \Cov(Y_0,Y_k)=\frac{1}{4(-3)^k} \quad \forall \, k \in \mathbb{N}.\]
				Let $V=\{a,b,c,d\}$ be the set of vertices of $K_4$. Without loss of generality, we may assume that the balanced labelling that we consider satisfies $\Val(a)=\Val(b)=1$ and $\Val(c)=\Val(d)=0$. One can check that the transition matrix of the random walk on $K_4$ can be diagonalized as
				\[
				\begin{pmatrix}
					0 & 1/3 & 1/3 & 1/3 \\
					1/3 & 0 & 1/3 & 1/3 \\
					1/3 & 1/3 & 0 & 1/3	\\
					1/3 & 1/3 & 1/3 & 0 \\
				\end{pmatrix}
				=
				\begin{pmatrix}
					-1 & -1 & -1 & 1 \\
					0 & 0 & 1 & 1 \\
					0 & 1 & 0 & 1	\\
					1 & 0 & 0 & 1 \\
				\end{pmatrix}
				\begin{pmatrix}
					\frac{-1}{3} & 0 & 0 & 0 \\
					0 & \frac{-1}{3} & 0 & 0 \\
					0 & 0 & \frac{-1}{3} & 0	\\
					0 & 0 & 0 & 1 \\
				\end{pmatrix}
				\begin{pmatrix}
					-1/4 & -1/4 & -1/4 & 3/4 \\
					-1/4 & -1/4 & 3/4 & -1/4 \\
					-1/4 & 3/4 & -1/4 & -1/4	\\
					1/4 & 1/4 & 1/4 & 1/4 \\
				\end{pmatrix}
				.\]
				This shows that $K_4$ is a $\frac{1}{3}$-expander graph. Moreover, using the above diagonalization one can compute powers of the transition matrix to get 
				\[ P^k(a,a)= \frac{3^k +3\cdot(-1)^k}{4 \cdot3^k} \quad \mbox{ and } \quad P^k(a,b)=\frac{3^k - (-1)^k}{4\cdot3^k} \quad \forall \, k \in \mathbb{N}.\]
				Therefore, for any $k \in \mathbb{N}$ we have
				\[ \mathbb{E}_\pi(Y_0 Y_k) = \prob{Y_0=1, Y_k=1}=\frac{1}{2} \prob{Y_k=1 | Y_0=1} = \frac{1}{2} (P^k(a,a) + P^k(a,b))= \frac{1}{4}\left (1+\frac{1}{(-3)^k}\right ).\]
				The claim follows from that fact that $\mathbb{E}_\pi(Y_0)\mathbb{E}_\pi(Y_k)=\frac{1}{4}$. In view of (\ref{eq:var of sum}), adding the covariances gives the result since
				\[ 2\sum_{i<j} \Cov(Y_i,Y_j) = 2\sum_{k=1}^{t-1} (t-k)\Cov(Y_0,Y_k) = \frac{1}{2} \sum_{k=1}^{t-1} (t-k)\frac{1}{(-3)^k} = -\frac{1}{8}t + O(1).\qedhere\]
		\end{proof}
	
		Example \ref{ex: uniform} also follows from computing the covariances and applying the formula (\ref{eq:var of sum}) for the variance of a sum of random variables.
		
		\begin{proof}[Proof of Example \ref{ex: uniform}]
			
			Select a labelling $\Val$ uniformly at random from all balanced labellings on $G$, and consider $Z_t= \sum_{i=0}^{t-1} \Val(X_i)$. The variance of $Z_t$, which is taken with respect to the uniform distribution on the set of all balanced labellings on $G$, will be denoted by $\Var(Z_t)$. When the labelling $\Val$ is fixed, we will write $\Var(Z_t|\Val)$ for its variance. We claim that 
			\[ \Var(Z_t) \geq \frac{t}{4} + \frac{1}{2}\left (\frac{1}{d} - \frac{3}{n-1}\right )t + O(1).\]
			Observe that the law of total variance tells us that
			\[ \Var(Z_t)= \mathbb{E}(\Var(Z_t|\Val)) + \Var(\mathbb{E}(Z_t|\Val)).\]
			Since the initial distribution of $(X_i)$ is the stationary one, we have that $\mathbb{E}(Z_t|\Val)=t/2$ for any balanced labelling on $G$, whence $\Var(\mathbb{E}(Z_t|\Val))=0$. Therefore, the claim implies
			\[ \mathbb{E}(\Var(Z_t|\Val)) \geq \frac{t}{4} + \frac{1}{2}\left (\frac{1}{d} - \frac{3}{n-1}\right )t + O(1),\]
			thus there must be a balanced labelling $\Val$ for which $\Var(Z_t|\Val)$ is greater or equal than the right hand side. To prove the claim, we first calculate $\mathbb{E}(Y_0 Y_k)$, where the expectation is taken with respect to the uniform distribution on the set of all balanced labellings on $G$, and $Y_i=\Val(X_i)$. Since the labeling is selected uniformly at random, observe that
			\begin{align*} \mathbb{E}(Y_0 Y_k) &= \prob{Y_k=1, Y_0=1|X_0=X_k} \prob{X_k=X_0} + \prob{Y_k=1, Y_0=1|X_0 \neq X_k} \prob{X_k \neq X_0}\\
						&=\frac{1}{2} \prob{X_k=X_0} + \frac{1}{2} \frac{\frac{n}{2} - 1}{n-1} \prob{X_k \neq X_0} = \frac{1}{2} \prob{X_k=X_0} + \frac{1}{4}\left (1-\frac{1}{n-1}\right ) \prob{X_k \neq X_0}\\
						&= \frac{1}{4} + \frac{1}{4} \left ( \prob{X_k = X_0} - \frac{1}{n-1}  \prob{X_k \neq X_0}\right )= \frac{1}{4} + \frac{1}{4} \left ( \frac{n}{n-1}\prob{X_k = X_0} - \frac{1}{n-1} \right ). 
			\end{align*}
			Recall that $\mathbb{E}(Y_0)=\mathbb{E}(Y_k)=\frac{1}{2}$, which gives
			\[ \Cov(Y_0,Y_k)= \frac{1}{4} \left ( \frac{n}{n-1}\prob{X_k = X_0} - \frac{1}{n-1} \right ).\]
			Adding all covariances gives
			\begin{align*}
				 \sum_{i<j} \Cov(Y_i,Y_j) &= \frac{1}{4} \frac{n}{n-1}\sum_{k=1}^{t-1} (t-k) \prob{X_k=X_0} - \frac{1}{4}\sum_{k=1}^{t-1} (t-k) \frac{1}{n-1}\\
				&= \frac{1}{4}\frac{n}{n-1} \sum_{k=1}^{t-1} (t-k) \prob{X_k=X_0} - \frac{t(t-1)}{8(n-1)} .
			\end{align*}
			Write $P$ for the transition matrix of $(X_i)$. We can use the spectral expansion of $P$ to calculate $\prob{X_k=X_0}$. Let $(\lambda_j)$ be the eigenvalues of $P$ and let $(f_j)$ be an orthonormal basis of $(\mathbb{R}^V, \langle \cdot, \cdot \rangle_{\pi})$ corresponding to $(\lambda_j)$. Lemma 12.2 in \cite{levin2017markov} gives for any $k \in \mathbb{N}$
			\[ P^k(x,y)= \sum_{j=1}^{n} f_j(x)f_j(y) \lambda_j^k \pi(y) \quad \forall \, x,y \in V,\]
			from where we deduce that
			\begin{align*}
				\prob{X_k=X_0}&= \frac{1}{n} \sum_{x \in V} P^k(x,x) = \frac{1}{n} 	\sum_{x \in V} \sum_{j=1}^n f_j(x)^2\lambda_j^k \pi(x)
				= \frac{1}{n} \sum_{j=1}^n \lambda_j^k\sum_{x \in V}  f_j(x)^2 \pi(x) 	=\frac{1}{n} \sum_{j=1}^n \lambda_j^k.
			\end{align*}
			Since $\lambda_1=1$, it is clear that $\prob{X_k=X_0}\geq \frac{1}{n}$ for any even $k \in \mathbb{N}$. Moreover, for any even $k \in \mathbb{N}$ we have
			\[ \prob{X_k=X_0} + \prob{X_{k+1}=X_0} = \frac{1}{n} \sum_{j=1}^n \lambda_j^k + \lambda_j^{k+1} = \frac{1}{n} \sum_{j=1}^n \lambda_j^k (1+\lambda_j)\geq \frac{2}{n}.\]
			Therefore,
			\begin{align*}\sum_{k=1}^{t-1} (t-k) \prob{X_k=X_0} &\geq (t-2) \prob{X_2=X_0}+ \frac{1}{n} \sum_{k=4}^{t-1} (t-k)= (t-2)\frac{1}{d}+ \frac{1}{n} \left (\frac{t(t-1)}{2} - 3t+6\right ) .
			\end{align*}
			Consequently,
			\[ \sum_{i<j} \Cov(Y_i,Y_j)\geq \frac{1}{4}\left ((t-2)\frac{1}{d} +\frac{1}{n-1} \left (\frac{t(t-1)}{2} - 3t+6\right ) \right ) - \frac{t(t-1)}{8(n-1)}=\frac{1}{4} \left (\frac{1}{d} - \frac{3}{n-1}\right )t + O(1). \]
			The claim now follows from the formula (\ref{eq:var of sum}) for the variance of a sum of random variables.
		\end{proof}
	
		\noindent \textbf{Acknowledgment:\ } We are grateful to Professor Fedor Nazarov for helpful suggestions that led to a sharper form of Theorem \ref{theo: local CLT Markov}.

		\bibliographystyle{plain}	\bibliography{A_local_central_limit_theorem_for_random_walks_on_expander_graphs}	

\begin{thebibliography}{10}

\bibitem{chiclana2022av1}
Rafael Chiclana and Yuval Peres.
\newblock A local central limit theorem for random walks on expander graphs,
  2022.
\newblock Preprint available on \url{arXiv:2212.00958v1}.

\bibitem{cohen2022expander}
Gil Cohen, Dor Minzer, Shir Peleg, Aaron Potechin, and Amnon Ta-Shma.
\newblock Expander random walks: the general case and limitations.
\newblock In {\em 49th {I}nternational {C}olloquium on {A}utomata, {L}anguages,
  and {P}rogramming ({ICALP} 2022)}, volume 229 of {\em LIPIcs. Leibniz Int.
  Proc. Inform.}, pages Paper No. 43, 18. Schloss Dagstuhl. Leibniz-Zent.
  Inform., Wadern, 2022.

\bibitem{cohen2021expander}
Gil Cohen, Noam Peri, and Amnon Ta-Shma.
\newblock Expander random walks: a {F}ourier-analytic approach.
\newblock In {\em S{TOC} '21---{P}roceedings of the 53rd {A}nnual {ACM}
  {SIGACT} {S}ymposium on {T}heory of {C}omputing}, pages 1643--1655. ACM, New
  York, [2021] \copyright 2021.

\bibitem{durret2019probability}
Rick Durrett.
\newblock {\em Probability---theory and examples}, volume~49 of {\em Cambridge
  Series in Statistical and Probabilistic Mathematics}.
\newblock Cambridge University Press, Cambridge, 2019.
\newblock Fifth edition of [ MR1068527].

\bibitem{gillman1998a}
David Gillman.
\newblock A {C}hernoff bound for random walks on expander graphs.
\newblock {\em SIAM J. Comput.}, 27(4):1203--1220, 1998.

\bibitem{golowhich2022a}
Louis Golowich.
\newblock A {N}ew {B}erry-{E}sseen {T}heorem for {E}xpander {W}alks, 2022.

\bibitem{golowich2022pseudorandomness}
Louis Golowich and Salil Vadhan.
\newblock {Pseudorandomness of Expander Random Walks for Symmetric Functions
  and Permutation Branching Programs}.
\newblock In Shachar Lovett, editor, {\em 37th Computational Complexity
  Conference (CCC 2022)}, volume 234 of {\em Leibniz International Proceedings
  in Informatics (LIPIcs)}, pages 27:1--27:13, Dagstuhl, Germany, 2022. Schloss
  Dagstuhl -- Leibniz-Zentrum f{\"u}r Informatik.

\bibitem{guruswami2021pseudobinomiality}
Venkatesan Guruswami and Vinayak~M. Kumar.
\newblock Pseudobinomiality of the sticky random walk.
\newblock In {\em 12th {I}nnovations in {T}heoretical {C}omputer {S}cience
  {C}onference}, volume 185 of {\em LIPIcs. Leibniz Int. Proc. Inform.}, pages
  Art. No. 48, 19. Schloss Dagstuhl. Leibniz-Zent. Inform., Wadern, 2021.

\bibitem{helson2010harmonic}
Henry Helson.
\newblock {\em Harmonic analysis}, volume~7 of {\em Texts and Readings in
  Mathematics}.
\newblock Hindustan Book Agency, New Delhi, second edition, 2010.

\bibitem{hoory2006expander}
Shlomo Hoory, Nathan Linial, and Avi Wigderson.
\newblock Expander graphs and their applications.
\newblock {\em Bull. Amer. Math. Soc. (N.S.)}, 43(4):439--561, 2006.

\bibitem{kolmogorov1949a}
A.~N. Kolmogorov.
\newblock A local limit theorem for classical {M}arkov chains.
\newblock {\em Izvestiya Akad. Nauk SSSR. Ser. Mat.}, 13:281--300, 1949.

\bibitem{levin2017markov}
David~A. Levin and Yuval Peres.
\newblock {\em Markov chains and mixing times}.
\newblock American Mathematical Society, Providence, RI, 2017.
\newblock Second edition of [ MR2466937], With contributions by Elizabeth L.
  Wilmer, With a chapter on ``Coupling from the past'' by James G. Propp and
  David B. Wilson.

\bibitem{Mann1996Berry}
Brad~W. Mann.
\newblock {\em Berry-{E}sseen central limit theorems for {M}arkov chains}.
\newblock ProQuest LLC, Ann Arbor, MI, 1996.
\newblock Thesis (Ph.D.)--Harvard University.
  \url{https://www.proquest.com/docview/304301471}.

\bibitem{nagaev1957some}
S.~V. Nagaev.
\newblock Some limit theorems for stationary {M}arkov chains.
\newblock {\em Teor. Veroyatnost. i Primenen.}, 2:389--416, 1957.

\bibitem{nagaev1962more}
S.~V. Nagaev.
\newblock More exact statement of limit theorems for homogeneous {Markov}
  chains.
\newblock {\em Theory Probab. Appl.}, 6:62--81, 1962.

\bibitem{penrose2011local}
Mathew~D. Penrose and Yuval Peres.
\newblock Local central limit theorems in stochastic geometry.
\newblock {\em Electron. J. Probab.}, 16:no. 91, 2509--2544, 2011.

\bibitem{vadhan2009expander}
Salil Vadhan.
\newblock Expander {G}raphs, 2009.
\newblock Lecture notes, Chapter 4.
  \url{https://people.seas.harvard.edu/~salil/cs225/spring09/lecnotes/Chap4.pdf}.

\bibitem{zolotukhin2018on}
Anatolii Zolotukhin, Sergei Nagaev, and Vladimir Chebotarev.
\newblock On a bound of the absolute constant in the {B}erry-{E}sseen
  inequality for i.i.d. {B}ernoulli random variables.
\newblock {\em Mod. Stoch. Theory Appl.}, 5(3):385--410, 2018.

\end{thebibliography}
	\end{document}